\documentclass[11pt,twoside,a4paper]{article}
\title{} \author{} \date{}
\usepackage{latexsym,amssymb,times}
\input amssym.def
\newtheorem{te}{Theorem}[section]
\newtheorem{prop}[te]{Proposition}
\newtheorem{cor}[te]{Corollary}
\newtheorem{fac}[te]{Fact}
\newtheorem{cla}[te]{Claim}
\newtheorem{lem}[te]{Lemma}

\newtheorem{rem}[te]{Remark}
\newtheorem{ex}[te]{Example}
\newtheorem{que}[te]{Question}
\def\dok{\noindent{\bf Proof. }}
\def\kdok{\hfill $\Box$ \par \vspace*{2mm} }
\def\a{\alpha}

\def\f{\varphi}
\def\p{\psi}
\def\o{\omega}
\def\k{\kappa}
\def\l{\lambda}

\def\r{\rho}
\def\s{\sigma}

\def\t{\tau}

\def\ve{\varepsilon}

\def\S{{\mathbb S}}
\def\T{{\mathbb T}}

\def\B{{\mathbb B}}
\def\N{{\mathbb N}}
\def\X{{\mathbb X}}
\def\Y{{\mathbb Y}}
\def\Z{{\mathbb Z}}

\def\A{{\mathbb A}}

\def\BL{{\mathbb L}}

\def\BK{{\mathbb K}}
\def\BI{{\mathbb I}}

\def\BJ{{\mathbb J}}

\def\CB{{\mathcal B}}

\def\CC{{\mathcal C}}
\def\CD{{\mathcal D}}
\def\CT{{\mathcal T}}

\def\M{{\mathcal M}}

\def\c{{\mathfrak{c}}}
\def\la{\langle}
\def\ra{\rangle}

\def\Sym{\mathop{\rm Sym}\nolimits}
\def\Sent{\mathop{\rm Sent}\nolimits}
\def\Form{\mathop{\rm Form}\nolimits}
\def\Mod{\mathop{\rm Mod}\nolimits}
\def\Ord{\mathop{\rm Ord}\nolimits}

\def\Th{\mathop{\rm Th}\nolimits}
\def\Lit{\mathop{\rm Lit}\nolimits}

\def\bcd{\dot{\bigcup}}
\def\du{\mathrel{\dot{\cup}}}
\def\rr{\mathrel{\rho}}
\def\Con{\mathop{\rm Con}\nolimits}
\begin{document}
\thispagestyle{plain}
\begin{center}
           {\large \bf \uppercase{Vaught's conjecture for sums of products of rooted trees}}
\end{center}
\begin{center}
{\bf Milo\v s S.\ Kurili\'c}\footnote{Department of Mathematics and Informatics, Faculty of Sciences, University of Novi Sad,
                                      Trg Dositeja Obradovi\'ca 4, 21000 Novi Sad, Serbia.
                                      e-mail: milos@dmi.uns.ac.rs}
\end{center}
\begin{abstract}
\noindent
For a class ${\mathcal C}$ of partial orders $\langle  {\mathcal C}\rangle _\Pi$ denotes its closure under finite direct products and isomorphism
and $\langle \langle  {\mathcal C}\rangle _\Pi\rangle _\Sigma$ denotes the closure of $\langle  {\mathcal C}\rangle _\Pi$ under lexicographic sums indexed by finite posets ${\mathbb I}$.
$\langle \langle  {\mathcal C}\rangle _\Pi\rangle _\Sigma$ is the class of partial orders isomorphic to sums ${\mathbb X}=\sum _{{\mathbb I}}\prod _{j<m_i}{\mathbb X}_{i,j}$,
where the factors ${\mathbb X}_{i,j}$ belong to ${\mathcal C}$.
We confirm Vaught's conjecture for each partial order ${\mathbb X}$
from the closure $\langle \langle  {\mathcal C}^{\rm rt}\rangle _\Pi\rangle _\Sigma$ of the class ${\mathcal C}^{\rm rt}$ of rooted trees.
In addition, defining ${\mathcal T} :=\mathop{\rm Th}\nolimits ({\mathbb X})$ and ${\mathcal T}_{i,j} :=\mathop{\rm Th}\nolimits ({\mathbb X}_{i,j})$, for all indices,
we show that: (a) ${\mathcal T}$ is $\omega $-categorical iff all the theories ${\mathcal T}_{i,j}$ are $\omega $-categorical;
(b) ${\mathcal T}$ satisfies VC$^\sharp$ (that is, $I({\mathcal T})\in \{1,{\mathfrak{c}}\}$), if VC$^\sharp$ holds for all theories ${\mathcal T}_{i,j}$.
For the partial orders ${\mathbb X}$ from the closure $\langle \langle  {\mathcal C}^{\rm rt}_{\rm fa}\rangle _\Pi\rangle _\Sigma$
of the class ${\mathcal C}^{\rm rt}_{\rm fa}$ of finitely axiomatizable rooted trees
we have more: ${\mathcal T}$ is finitely axiomatizable and
${\mathbb Y}\in \mathop{\rm Mod}\nolimits ({\mathcal T})$
iff ${\mathbb Y} \cong \sum _{{\mathbb I}}\prod _{j<m_i}{\mathbb Y}_{i,j}$, where ${\mathbb Y}_{i,j}\in \mathop{\rm Mod}\nolimits ({\mathcal T}_{i,j})$, for all indices.
If, in particular, the factors ${\mathbb X}_{i,j}$ are sums of finitely axiomatizable linear orders $\sum _{{\mathbb K} _{i,j}}{\mathbb L} _{i,j,k}$,
then all these results are expressible in terms of the theories ${\mathcal T} _{i,j,k}:=\mathop{\rm Th}\nolimits ({\mathbb L} _{i,j,k})$ of linear order.

As a by-product we prove that for each $n$ the class ${\mathcal C} _n$ of partial orders isomorphic to a direct product of $n$ rooted trees of size $>1$
is first-order definable.
Moreover, the same holds for each subclass ${\mathcal C}$ of ${\mathcal C} _n$ isolated by a set of first-order properties of the factors ${\mathbb X}_i$ of a product $\prod _{i<n}{\mathbb X} _i$
(e.g.\ we can request that the factors are linear orders with some properties, trees of width $\leq n$, etc.).
Another by-product is that each $\omega $-categorical partial order from the class $\langle  \langle {\mathcal C}^{\rm rt}_{\rm fb}\rangle _\Pi\rangle _\Sigma$,
where ${\mathcal C}^{\rm rt}_{\rm fb}$ is the class of finite-branching rooted trees, is finitely axiomatizable.
This result is related to the results of Rosenstein (for the class of linear orders) and Schmerl (for the class of partial orders of finite width).


\end{abstract}
\section{Introduction}\label{S1}
Vaught's conjecture (VC), stated by Robert Vaught \cite{Vau},
is the statement that the number $I(\CT ,\o)$ of non-isomorphic countable models of a complete countable first-order theory $\CT$
is either at most countable or continuum.
The full conjecture follows from its restriction to the class of partial orders (see \cite{Hodg}, p.\ 231)
and it was confirmed for the theories of
linear orders (Matatyahu Rubin \cite{Rub}), model-theoretic trees (John Steel \cite{Stee}), reticles (James Schmerl \cite{Sch3}),
Boolean algebras (Paul Iverson \cite{Ive}).

Using the concept exploited in \cite{KMon}--\cite{Kprod1} we deal with the following question:
If VC is confirmed for the structures from a class $\CC$,
is it true for the structures belonging to its closure $\la\CC \ra$,
obtained by some model-theoretic constructions?
Clearly, the same question can be asked for the ``sharp" version of the conjecture, VC$^\sharp$, saying that $I(\CT,\o)\in \{1,\c\}$.
For example,  using Rubin's results from  \cite{Rub},
VC$^\sharp$ was confirmed  in \cite{KFMD}
for the relational structures definable by quantifier free formulas in labeled linear orders,
and in \cite{Ksharp} for the partial orders from the closure $\la\CC ^{\rm lo}\ra_{\du\Pi}$ of the class of linear orders under finite direct products and disjoint unions.
On the basis of Steel's results from  \cite{Stee} VC was confirmed in \cite{Kprod1}
for the partial orders from the closure $\la\CC ^{\rm rt}\ra_{\du\Pi}$, where $\CC ^{\rm rt}$ is the class of rooted trees.

Generally speaking, if $\CC$ is a class of partial orders,
$\la \CC\ra_\Pi$ its closure under finite direct products and isomorphism
and $\la\la \CC\ra_\Pi\ra_\Sigma$ the closure of $\la \CC\ra_\Pi$ under finite lexicographic sums,
then each partial order from $\la\la \CC\ra_\Pi\ra_\Sigma$ is isomorphic to one of the form $\X=\sum _{\BI}\prod _{j<m_i}\X_{i,j}$,
where $\BI$ is a finite partial order and $\X_{i,j}\in \CC$, for all $j<m_i$ and $i\in I$ (see Fact \ref{T944}).
Defining $\CT :=\Th (\X)$ and $\CT_{i,j} :=\Th (\X_{i,j})$, for $j<m_i$ and $i\in I$,
we regard the interplay between these theories and
consider the following properties:

(P1) $(\forall i\in I \;\forall j<m_i \;\;\CT _{i,j} \mbox{ satisfies VC})\Rightarrow \CT  \mbox{ satisfies VC}$,

(P2) $(\forall i\in I \;\forall j<m_i \;\;\CT _{i,j} \mbox{ satisfies VC$^\sharp$})\Rightarrow \CT  \mbox{ satisfies VC$^\sharp$}$,

(P3) $(\forall i\in I \;\forall j<m_i \;\;\CT _{i,j} \mbox{ is $\o$-categorical})\Leftrightarrow \CT  \mbox{ is $\o$-categorical}$,

(P4) $(\forall i\in I \;\forall j<m_i \;\;\CT _{i,j} \mbox{ is finitely axiomatizable})\Rightarrow \CT  \mbox{ is finitely axiomatizable}$,

(P5) $\Mod (\CT)=\{\Y : \Y \cong \sum _{\BI}\prod _{j<m_i}\Y_{i,j}, \mbox{ where } \Y_{i,j}\in \Mod (\CT_{i,j}),\mbox{ for all } j<m_i \mbox{ and }i\in I \}$.

\noindent
In this paper we consider the closure $\la\la \CC^{\rm rt}\ra_\Pi\ra_\Sigma$ of the class $\CC^{\rm rt}$ of rooted trees.

Our first result is that Vaught's conjecture is true
for each partial order $\X=\sum _{\BI}\prod _{j<m_i}\X_{i,j}$ belonging to the closure $\la\la \CC^{\rm rt}\ra_\Pi\ra_\Sigma$
(so we have (P1)) and, in addition, that (P2) and (P3) are true as well.
The proof is based on the following two known results:
1. VC holds for finite products of rooted trees \cite{Kprod1};
thus, we have VC for the summands $\X _i:=\prod _{j<m_i}\X_{i,j}$, $i\in I$,
and, since $\min \X _i$ exists for each $i\in I$, the sum $\X=\sum _{\BI}\X_i$ is an FLD$_0$-poset (see Section \ref{S2} for definitions);
2. VC holds for almost Vaught's FLD$_0$-theories \cite{KFLD}.
So, it remains to be proved that the FLD$_0$-theory $\Th (\X)$ is almost Vaught's,
and for this aim in Section \ref{S3} we prove that for each $n\in \N$ the class
$$
\CC _n :=\{ \X \in \Mod _{L_b}:\X \mbox{ is isomorphic to a direct product of $n$ rooted trees of size $>1$}\}
$$
is first-order definable.
The theorem confirming (P1)--(P3) for each partial order from the closure $\la\la \CC^{\rm rt}\ra_\Pi\ra_\Sigma$ is proved in Section \ref{S4}.

The following two sections are devoted to our next goal: to confirm (P1)--(P5) for the partial orders from the closure $\la\la \CC^{\rm rt}_{\rm fa}\ra_\Pi\ra_\Sigma$
of the class $\CC^{\rm rt}_{\rm fa}$ of finitely axiomatizable rooted trees.

So, in Section \ref{S5} we extend the main result of Section \ref{S3}
showing that, roughly, each subclass $\CC$ of $\CC _n$
isolated by a set of first-order properties of the factors $\X_i$  of a product $\prod _{i<n}\X _i$ is first-order definable.
For example, the class of distributive lattices isomorphic to a direct product of $n$ linear orders with a smallest element is first-order definable;
the same holds if we take linear orders with additional first-order properties (e.g. discrete, dense, bounded, etc.).

Section \ref{S6} deals with FLD theories having a model $\sum _{\BI}\X_i$ with finitely axiomatizable summands $\X _i$.
The properties analogous to (P1)--(P5) are confirmed.

The theorem confirming (P1)--(P5) for partial orders from $\la\la \CC^{\rm rt}_{\rm fa}\ra_\Pi\ra_\Sigma$ is proved in Section \ref{S7}
and as a by-product we obtain a result related to a general question: Which $\o$-categorical theories are finitely axiomatizable?
In our context, a more specific task is to detect classes $\CC$ of partial orders such that
\begin{equation}\label{EQ985}
\forall \X \in \CC \;(\X \mbox{ is $\o$-categorical} \Rightarrow \X  \mbox{ is finitely axiomatizable}).
\end{equation}
We recall that (\ref{EQ985}) was confirmed for the class of linear orders by Rosenstein \cite{Rosen1}
and, more generally, for the class of partial orders of finite width by Schmerl \cite{Sch1}.
In addition, Schmerl  \cite{Sch0} proved that for an $\o$-categorical tree $\X$ we have: $\X$ is finitely axiomatizable iff $\X$ is finite-branching.
So, denoting by $\CC^{\rm rt}_{\rm fb}$ the class of finite-branching rooted trees,
we confirm (\ref{EQ985}) for each partial order $\X=\sum _{\BI}\prod _{j<m_i}\X_{i,j}$ belonging to the closure $\la \la\CC^{\rm rt}_{\rm fb}\ra_\Pi\ra_\Sigma$
(here $\X_{i,j}\in \CC^{\rm rt}_{\rm fb}$, for all indices).

Section \ref{S8} provides a class of examples.
Namely, while we do not have a clear picture of the class $\CC ^{\rm rt}_{\rm fa}$  of finitely axiomatizable rooted trees,
the class of finitely axiomatizable linear orders is much better understood
(e.g.\ it contains the ordinals $<\o ^\o$, their inverses,
and, by the work of  Rubin \cite{Rub}, the smallest class of linear orders which contains 1 and which is closed under addition, $\Z$-sums and finite shuffles).
So if  $\CC$ is the class of lexicographic sums $\sum _\BK \BL _k$, where $\BK$ is a finite rooted tree
and $\BL _k$, $k\in K$, are finitely axiomatizable linear orders with a smallest element,
then  by the results of Section \ref{S6} we have $\CC \subset \CC ^{\rm rt}_{\rm fa}$
and for $\X = \sum _{\BI }\prod _{j<m_i}\sum _{\BK_{i,j}}\BL _{i,j,\,k} \in \la\la\CC\ra_\Pi\ra_\Sigma$
we have (P1)--(P5) expressed in terms of the theories $\CT_{i,j,\,k}:=\Th (\BL _{i,j,\,k}) $.
For example, $\Y$ is a model of $\Th (\X)$ iff
$\Y \cong \sum _{\BI }\prod _{j<m_i}\sum _{\BK_{i,j}}\BL _{i,j,\,k}'$, where $\BL _{i,j,\,k}'\in \Mod (\CT _{i,j,\,k})$, for all indices $i,j,k$.
\section{Preliminaries}\label{S2}
For a language $L$ by $\Mod _L$, $\Form _L$ and $\Sent _L$ we denote the class of $L$-structures, the set of first-order $L$-formulas and the set of $L$-sentences.
In particular, $L_b$ denotes the language with one binary relational symbol; working with partial orders we will take $L_b=\la \leq\ra$.
Isomorphism and elementary equivalence are denoted by $\cong$ and $\equiv$.
$I(\CT, \o):= |\Mod (\CT, \o)/\!\cong |$ is the number of non-isomorphic countable models of a complete theory $\CT$;
if $\CT$ has a finite model, $\X$, then, for convenience, we define $\Mod (\CT, \o)/\!\cong \;:= \{ [\X]\}$ and $I(\CT, \o):=1$.
For simplicity, instead of $I(\CT, \o)$ and $I(\Th (\X), \o)$ we will write only $I(\CT)$ and $I(\X )$.

Recall that a class $\CC\subset \Mod _L$  is called {\it first-order definable} (or a {\it basic elementary class}, \cite{CK})
iff there is a sentence $\s \in \Sent _L$ such that $\X\in \CC$ iff $\X \models \s$, for each $\X\in \Mod _L$.
A theory $\CT \subset \Sent _L$ is called {\it finitely axiomatizable} iff there is a sentence $\s \in \Sent _L$ such that $\X\models \CT$ iff $\X \models \s$, for each $\X\in \Mod _L$.
A structure $\X\in \Mod _L$ will be called finitely axiomatizable iff $\Th (\X)$ is finitely axiomatizable.

A partial order $\X$ is a (model-theoretic) {\it tree} iff $(\cdot ,x]:=\{ y\in X: y\leq x\}$ is a linear order, for each $x\in X$.
If, in addition, $\X$ has a smallest element, $\min \X$, then $\X$ ia a {\it rooted tree}.
\paragraph{Substructures with parametrically definable domains}
If $\f (w_0,\dots,w_{m-1},v)=\f (\bar w,v)$ is an $L_b$-formula,
then the corresponding {\it relativization} of an $L_b$-formula $\p (v_0\dots v_{n-1})=\p (\tilde v)$, where $n\in\o$,
is the $L_b$-formula $\p^{\f }(\bar w,\tilde v)$ defined by recursion in the following way:
$\p^{\f }(\bar w,\tilde v)                                        :=  \p (\tilde v)$, if $\p (\tilde v)$ is atomic,
$(\neg \p (\tilde v))^{\f }(\bar w,\tilde v)                      :=  \neg \p^{\f } (\bar w,\tilde v)$,
$(\p_0 (\tilde v)\land \p_1 (\tilde v))^{\f }(\bar w,\tilde v)    :=  \p_0^{\f } (\bar w,\tilde v)\land \p_1^{\f } (\bar w,\tilde v)$,
and similarly for disjunction,
$(\forall u \;\p (\tilde v,u))^{\f }(\bar w,\tilde v)             :=  \forall u\; (\f (\bar w,u)\Rightarrow \p ^{\f }(\bar w,\tilde v,u))$,
and $(\exists u \;\p (\tilde v,u))^{\f }(\bar w,\tilde v)             :=  \exists u\; (\f (\bar w,u)\land \p ^{\f }(\bar w,\tilde v,u))$.
If $\X \in \Mod _{L_b}$ and $\bar a \in X^m$, then $D_{\f (\bar a,v),\X}:=\{ x\in X: \X\models \f [\bar a,x]\}$ is the corresponding {\it parametrically definable substructure of $\X$}.
The following claim is folklore (for a proof see e.g.\ Fact 2.3 of \cite{KFLD}).
\begin{fac}\label{T704}
If $\f (\bar w,v)\in \Form _{L_b}$, $\X \in \Mod _{L_b}$, $\bar a \in X^m$ and $D_{\f (\bar a,v),\X}\neq \emptyset$,
then for each $L_b$-sentence $\p$ we have $D_{\f (\bar a,v),\X}\models \p  $ iff $\;\X \models \p^{\f }[\bar a]$.
\end{fac}
\begin{fac}\label{T922}
If $\f (v)\in \Form _{L_b}$, $\X,\Y \in \Mod _{L_b}$, $D_{\f (v),\X}\neq \emptyset$ and $\X \equiv \Y$, then $D_{\f (v),\X}\equiv D_{\f (v),\Y}$.
\end{fac}
\dok
By Fact \ref{T704} for $\p\in \Sent_{L_b}$ we have:
$D_{\f (v),\X}\models \p$
iff $\X\models \p ^\f$
iff $\Y\models \p ^\f$
iff $D_{\f (v),\Y}\models \p$.
\hfill $\Box$
\paragraph{Lexicographic sums}
Let $\BI=\la  I, \r _I\ra$ and $\X _i =\la X_i , \r _i \ra$, $i\in I$, be $L_b$-structures with pairwise disjoint domains.
The {\it lexicographic sum of the structures $\X _i$, $i\in I$, over the structure $\BI$}, in notation $\sum _{\BI}\X _i$, is the  $L_b$-structure
$\X :=\la X, \r \ra$, where $X := \bigcup _{i\in I}X_i$ and
for $x,x'\in X$ we have: $x\,\r \,x'$ iff
\begin{equation}\label{EQ200}
\exists i\in I \;\; (x,x'\in X_i \land x \,\r _i\, x') \;\lor\;
\exists \la i,j \ra \in \r _I \setminus \Delta _I \;\; (x\in X_i \land x'\in X_j ),
\end{equation}
where $\Delta _I:=\{ \la i,i\ra:i\in I\}$. For a proof of the following statement see e.g.\ Fact 2.7 of \cite{KFLD}.
\begin{fac}\label{T200}
If $\sum _{\BI}\X _i$ and $\sum _{\BI}\Y _i$ are lexicographic sums of $L_b$-structures,
then

(a) If $\X _i \cong \Y _i$, for $i\in I$, then $\sum _{\BI}\X _i \cong \sum _{\BI}\Y _i$;

(b) If $\X _i \equiv  \Y _i$, for all $i\in I$, then $\sum _{\BI}\X _i \equiv  \sum _{\BI}\Y _i $.
\end{fac}
$\CC^{\rm fin}$ will denote the class of finite partial orders and
if $\CC$ is an $\cong$-closed class of partial orders,
by $\la \CC \ra _{\Sigma}$ we denote the minimal closure of $\CC$ under isomorphism and finite lexicographic sums.\footnote{If $\BI \in \CC ^{\rm fin}$
and $\la \Y _i :i\in I\ra \in \CC ^I$,
then in $\CC$ there are structures $\X_i$, $i\in I$, with pairwise disjoint domains such that $\X_i \cong\Y _i $, for $i\in I$,
and the sum $\sum _\BI \X _i\in \la \CC \ra _{\Sigma}$ is well defined.
Taking another representatives $\X_i' \cong\Y _i $
by Fact \ref{T200}(a) we have $\sum _\BI \X _i'\cong \sum _\BI \X _i$;
so $\la \CC \ra _{\Sigma}$ can be regarded as a closure under finite lexicographic sums of order types from $\CC$.}
\begin{fac}\label{T945}
If $\CC$ is an $\cong$-closed class of partial orders, then
\begin{equation}\label{EQ981}\textstyle
\la \CC \ra _{\Sigma}  = \textstyle \{\sum _\BI \X _i:\BI \in \CC ^{\rm fin}\land  \la \X _i :i\in I\ra \in \CC ^I
                           \land \forall i,j\in I \; (i\neq j\Rightarrow X_i \cap X_j=\emptyset)\}.
\end{equation}
\end{fac}
\dok
Let $\CC _0$ denote the r.h.s.\ of (\ref{EQ981}). Taking sums over $\BI\cong 1$ we see that $\CC\subset \CC _0$.
The class $\CC _0$ is closed under finite lexicographic sums,
because by Fact 5.1 of \cite{KFLD}, roughly,
a sum of sums is a sum (i.e.\ $\sum _{\BI}\sum _{\BJ _i}\X _j=\sum _{\sum _{\BI}\BJ _i }\X _j$).
In order to show that $\CC _0$ is $\cong$-closed
we suppose that $\sum _\BI \X _i \in \CC _0$ and that $f:\sum _\BI \X _i\rightarrow \Y$ is an isomorphism.
Then defining $Y_i:=f[X_i]$, for $i\in I$, we obtain a partition $\{ Y_i :i\in I\}$ of $Y$;
also $\Y_i\cong \X _i$, for all $i\in I$, and, since $\CC$ is $\cong$-closed, $\Y_i\in \CC$, for $i\in I$.
By (\ref{EQ200}) we have $\Y =\sum _{\BI}\Y_i$ and, hence, $\Y\in \CC _0$.
If $\CC \subset \CD$ and the class $\CD $ is closed under isomorphism and finite lexicographic sums,
then, clearly,  $\CC _0 \subset \CD$. Thus, $\la \CC \ra _{\Sigma}=\CC _0$ indeed.
\hfill $\Box$
\paragraph{FLD theories}
By \cite{KFLD} a complete theory $\CT$ of partial order is an FLD$_0$-{\it theory}
iff some (equivalently, any) of its models $\X$
admits a finite lexicographic decomposition $\X =\sum _{\BI}\X _i$,
where $\BI$ is a finite partial order and $\X _i$-s are pairwise disjoint partial orders with a smallest element.
Then we write $\sum _{\BI}\X_i\in \CD (\CT)$;
namely, $\CD (\CT)$ denotes the class of all such decompositions of models of $\CT$.
Dually we define FLD$_1$-theories (here $\X _i$-s have a largest element).
By the dual of Theorem 3.1 from \cite{KFLD} we have
\begin{fac}\label{T700}
If $\CT$ is an FLD$_0$-theory and $\X =\sum _{\BI}\X_i\in \CD (\CT)$,
where $\BI=\la n,\leq _{\BI}\ra$, $r_i=\min \X _i$, for $i<n$, and $\bar r:=\la r_0,\dots ,r_{n-1}\ra$, then
there are $L_b$-formulas $\f_i (w_0,\dots ,w_{n-1},v)=\f_i (\bar w,v)$, $i<n$, such that\\[-7mm]
\begin{itemize}\itemsep=-1mm
\item[(a)] $X_i=D_{\f _i (\bar r , v), \X}$, for $i<n$, the formula $\ve (\bar w, u,v)=\bigvee _{i<n} (\f _i(\bar w , u)\land \f _i(\bar w , v))$
               defines in $\X$ an equivalence relation on the set $X$ and $X/D_{\ve (\bar r, u,v),\X}=\{ X_i :i<n\}$;
\item[(b)] If  $\,\Y  \equiv \X$ and $\t _i \in \Th (\X _i)$, for $i<n$,
               then there is $\bar r ':=\la r_0',\dots ,r_{n-1}'\ra \in Y^n$ such that defining $Y_i:=D_{\f _i (\bar r ' , v), \Y}$, for $i<n$, we have\\[-8mm]
               \begin{itemize}\itemsep=-1mm
               \item[(i)] $\{ Y_i:i<n\}$ is a partition of the set $Y$  and $Y/D_{\ve (\bar r',u,v),\Y} =\{ Y_i:i<n\}$,
               \item[(ii)] $\Y =\sum _{\BI}\Y_i$ and $r_i'=\min \Y _i$, for $i<n$,
               \item[(iii)] $\Y _i \models \t _i$, for $i<n$.
               \end{itemize}
\end{itemize}
\end{fac}
An FLD$_0$-theory $\CT$ is called {\it actually Vaught's}
iff for some decomposition $\sum _{\BI}\X_i\in \CD (\CT)$  there are sentences $\t _i\in \Th (\X _i)$, $i\in I$, providing VC; namely,
\begin{equation}\label{EQ741}\textstyle
\exists \sum _{\BI}\X_i\in \CD (\CT) \;\;\forall i\in I \;\;\exists \t _i \in \Th (\X _i) \;\;\forall \Z\models \t _i \;\;(I(\Z)\leq \o \lor I(\Z)=\c).
\end{equation}
By the dual Theorem 4.1 of \cite{KFLD} (proved for FLD$_1$-theories) we have
\begin{fac}\label{T712}
Vaught's conjecture is true for each actually Vaught's  FLD$_0$-theory $\CT$ and we have
\begin{equation}\label{EQ754}\textstyle
I(\CT)   = \left\{
           \begin{array}{cl}
           \c,         & \mbox{ if }\; \exists \sum _{\BI}\X_i\in \CD (\CT)\;\prod_{i\in I} I(\X _i)=\c, \;\mbox{ or $\CT$ is large}, \\[1mm]
           1,          & \mbox{ if }\; \exists \sum _{\BI}\X_i\in \CD (\CT)\;\prod_{i\in I} I(\X _i)=1,\\[1mm]
           \in [3,\o], & \mbox{ otherwise.}
           \end{array}
           \right.
\end{equation}
\end{fac}
\paragraph{Direct products} We will use the following well-known facts (valid for any language $L$, see \cite{FV}).
\begin{fac}\label{T043}
If $\X _i$ and $\Y _i$, for $i\in I$, are $L$-structures, then we have

(a) $\prod _{i\in I}\X _i \cong \prod _{i\in I}\X _{\pi (i)}$, for any permutation $\pi \in \Sym (I)$;

(b) If $\X _i \cong \Y _i$, for all $i\in I$, then $\prod _{i\in I}\X _i \cong \prod _{i\in I}\Y _i$;

(c) If $\X _i \equiv \Y _i$, for all $i\in I$, then $\prod _{i\in I}\X _i \equiv \prod _{i\in I}\Y _i$.
\end{fac}
If $\CC$ is an $\cong$-closed class of partial orders
by $\la \CC \ra_{\Pi}$ we denote the minimal closure of the class $\CC $ under finite direct products and isomorphism.
\begin{fac}\label{T944}
If $\CC$ is an $\cong$-closed class of partial orders, then

(a) $\la \CC \ra_{\Pi}=\{ \Y \in \Mod _{L_b}:\exists n\in \N \;\;\exists \X_0,\dots,\X_{n-1}\in \CC \;\; \Y \cong \prod _{i<n}\X_i \}$;

(b) $\Y\in \la\la \CC \ra_{\Pi} \ra _{\Sigma}$ iff there are $\BI\in \CC ^{\rm fin}$,
$n_i \in \N$ and $\X_{i,j}\in \CC $, for $j<n_i$ and $i\in I$, such that the sets $X_i:=\prod _{j<n_i}X_{i,j}$, $i\in I$, are pairwise disjoint and
\begin{equation}\label{EQ982}\textstyle
\Y\cong \sum _{\BI}\prod _{j<n_i}\X_{i,j}.
\end{equation}
\end{fac}
\dok
(a) Let $\CC '$ be the r.h.s.\ of the equality.
If $\X \in\CC $, then $\X =\prod _{i<1}\X_i$, where $\X_0=\X$; thus $\CC \subset\CC'$.
It is evident that the class  $\CC'$ is $\cong$-closed.
If  $\CC' \ni\Y ' \cong \prod _{i<m}\X_i' $ and $\CC' \ni\Y ''\cong \prod _{i<n}\X_i '' $,
then by Fact \ref{T043}(b) $\Y' \times \Y ''\cong \prod _{i<m+n}\X_i $, where $\X _i=\X_i'$, for $i<m$, and $\X _{m+i}=\X_i''$, for $i<n$.
Thus $\CC'$ is closed under products of two and (by induction) of finitely many factors.
If $\CC \subset \CD$, where  $\CD\subset \Mod _{L_b}$ is a class closed under finite direct products and isomorphism,
and if $\CC '\ni\Y \cong \prod _{i<n}\X_i $, where $\X_i\in \CC $, for $i<n$,
then $\Y\in \CD$.
Thus $\CC' \subset \CD$,
which shows the minimality of $\CC'$ and, hence $\la \CC \ra_{\Pi}=\CC'$.

(b) By (\ref{EQ981}) $\Y\in \la\la \CC \ra_{\Pi} \ra _{\Sigma}$
iff $\Y =\sum _\BI \Y _i$, where
$\BI \in \CC ^{\rm fin}$ and
$\Y _i \in  \la \CC \ra_{\Pi} $, for $i\in I$, and
$Y_i \cap Y_j=\emptyset$, for different $i, j\in I$.
By (a) for each $i\in I$ there is $n_i\in \N$
such that  $\Y _i \cong \X_i:=\prod _{j<n_i}\X_{i,j}$, where $\X_{i,j}\in \CC $, for $j<n_i$.
Since the class $\la \CC \ra_{\Pi}$ is $\cong$ closed
we can assume that $X_i \cap X_j=\emptyset$, for different $i, j\in I$;
so, the sum $\sum _\BI \X _i$ is well defined
and, by Fact \ref{T200}(a), $\Y =\sum _\BI \Y _i\cong \sum _\BI \X _i=\sum _{\BI}\prod _{j<n_i}\X_{i,j}$.
Conversely, under the assumptions we have $\Y\cong \sum _{\BI}\Y_i$,
where $\Y_i:=\prod _{j<n_i}\X_{i,j} \in \la \CC \ra_{\Pi}$, for $i\in I$.
Thus $\sum _{\BI}\Y_i \in \la\la \CC \ra_{\Pi} \ra _{\Sigma}$
and $\Y\in \la\la \CC \ra_{\Pi} \ra _{\Sigma}$,
since $\la\la \CC \ra_{\Pi} \ra _{\Sigma}$ is $\cong$-closed
(see Fact \ref{T945}).
\hfill $\Box$
\paragraph{Products of rooted trees}
By $\CC^{\rm rt}$ we denote the class of rooted trees. By Theorem 4.12 of \cite{Kprod1} we have
\begin{fac}\label{T544}
If $\X =\prod  _{i<n}\X_i$, where $\X _i\in\CC ^{\rm rt}$, for $i<n$, $\CT:=\Th (\X )$ and $\CT _i:=\Th (\X _i)$, $i<n$, then\\[-5mm]
\begin{itemize}\itemsep -0.5mm
\item[\rm (a)] $\Y \equiv \prod _{i<n} \X _i $ iff $\;\Y \cong\prod _{i<n} \Y _i$, where $\Y _i \equiv \X _i$, for each $i<n$;
\item[\rm (b)] The theory $\CT$ satisfies Vaught's conjecture:
               $I(\CT )=\k:=\prod _{i<n}I(\CT _i)$, if $\k\in \{1,\o,\c \}$,\\
               and $I(\CT ) \in [3,\o )$, otherwise.
\end{itemize}
\end{fac}
\section{Products of n rooted trees: first-order definability}\label{S3}
Our goal is to confirm VC for each partial order of the form $\X=\sum _{\BI}\prod _{j<m_i}\X_{i,j}$,
where $\BI$ is a finite partial order and  $m_i \in \N$ and $\X_{i,j}$, $j<m_i$, are rooted trees, for each $i\in I$.
Since the products $\X _i:=\prod _{j<m_i}\X_{i,j}$, $i\in I$, are partial orders with a smallest element
the theory $\CT :=\Th (\X)$ is an FLD$_0$-theory
and in order to apply Fact \ref{T712} it remains to be shown that $\CT$ is actually Vaught's (see (\ref{EQ741})).
Clearly $\sum _{\BI}\X _i \in \CD (\CT)$ and, by Fact \ref{T544}(b), VC is true for $\Th (\X _i)$, for each $i\in I$.
Thus, roughly speaking, if for each $n\in \N$ there is  a first order sentence $\t _n$ is saying ``I am a product of $n$ rooted trees",
then $\t _{m_i}\in \Th (\X _i)$ and $\t _{m_i}$ forces VC in the sense of (\ref{EQ741}), for $i\in I$; so, $\CT$ is actually Vaught's.

Clearly $\t _1$ exists.
But, strictly speaking (and working in ZFC),
for $n>1$ the class of products of $n$ rooted trees is not closed under isomorphism
and, hence, it is not first-order axiomatizable.
For example, the ordinal $\o$ is a rooted tree,
$\X :=\la \o\times \o, \leq ^\X\ra$ is a product of two rooted trees
and if $f:\o\times \o\rightarrow \o$ is a bijection,
and $\Y=\la \o \leq ^\Y\ra$, where $m\leq ^\Y n$ iff $f^{-1}(m)\leq ^\X f^{-1}(n)$, for $m,n\in \o$,
then, clearly, $\Y\cong \X$.
But $\Y$ is not a direct product of two rooted trees,
because its domain $\o$ is not a set of ordered pairs!
Namely, an ordered pair $\la x,y\ra :=\{\{x\},\{x,y\}\}$ is a non-empty set and $\emptyset \in \o$. So for $n\geq 2$ let
$$
\CC _n :=\{ \X \in \Mod _{L_b}:\X \mbox{ is isomorphic to a direct product of $n$ rooted trees of size $>1$}\}.
$$
Excluding the factors of size 1 in the definition of $\CC _n$ is not a restriction at all,
since $\prod _{i\in I}\X _i \cong \prod _{i\in J}\X _i$, if $J:=\{ i\in I : |X_i|>1\}\neq\emptyset$,
and if $J=\emptyset$, then the product is a trivial, one-element poset.

So, the rest of this section is devoted to a proof of the following statement.
\begin{te}\label{T915}
For each $n\geq 2$ there is an $L_b$-sentence $\t_{\CC _n}$ such that for each $L_b$-structure $\X$ we have
\begin{equation}\label{EQ958}
\X \in \CC _n \Leftrightarrow \X \models \t_{\CC _n}.
\end{equation}
\end{te}
\subsection{Detection of $\t_{\CC _n}$ and a proof of the implication ``$\Rightarrow$" in (\ref{EQ958})}
We recall that our language is $L_b:=\la \leq\ra$
and note that the sentence $\t_{\CC _n}$ will be a conjunction including a sentence $\f ^{\rm po}_{0}$
saying that the structure is a partial order with a smallest element.
Thus, our goal is to prove that $\CC_n =\Mod (\t_{\CC _n})$
and, since $\CC_n \cup \Mod (\t_{\CC _n}) \subset \Mod (\f ^{\rm po}_{0})$
and a smallest element is definable in all models of $\f ^{\rm po}_{0}$,
in our formulas we will use abbreviations $v>0$, $v=0$ etc.\ (which can be replaced by pure $L_b$-formulas).
Let $\CC ^{\rm rt}_{>1}$ denote the class of rooted trees of size $>1$.

We recall some facts from \cite{Kprod1}.
Let $\X = \prod _{i<n}\X_i$,
where $\X _i \in\CC ^{\rm rt}_{>1}$ and $0_i:=\min \X _i$, for $i<n$; then $\bar 0:=\la 0_0,\dots, 0_{n-1}\ra =\min \X$.
For $i<n$, let $X_i^+ :=X_i\setminus \{ 0_i\}$
and let $\{X_{i,j} :j\in J_i\}$ be the partition of the suborder $\X _i^+$ of $\X _i$ into its connectivity components.
Then $\X _i ^+ =\bcd _{j\in J_i}\X_{i,j}$ is a disjoint union of connected trees.
Defining $X^{(m)}=\{ \bar x\in \prod _{i<n}X_i: |\{ i<n:x_i>0_i \}|=m\}$, for $m\leq n$, we have
\begin{equation}\label{EQ589}\textstyle
X^{(m)}=\bigcup _{K\in[n]^m} \bigcup _{\bar j\in \prod _{i\in K}J_i}\Big\{ \bar x\in \prod _{i<n}X_i : \forall i\in K \;(x_i\in X_{i,j_i}) \land \forall i\in n\setminus K \;(x_i=0_i)\Big\}
\end{equation}
and $\{ X^{(m)}:0\leq m\leq n\}$ is a partition of the set $X$. In particular, $X^{(0)}=\{ \bar 0\}$ and for $m=1$
\begin{eqnarray}
X^{(1)} & = &\textstyle\bigcup _{i<n}\bigcup _{j\in J_i}S_{i,j},
              \;\mbox{ where } S_{i,j} :=\{ 0_0\} \times \dots \times \{ 0_{i-1}\} \times X_{i,j}\times \{ 0_{i+1}\} \times \dots \times \{ 0_{n-1}\},\label{EQ590}\\
        & = &\textstyle\bigcup _{i<n}A_i,   \;\;\mbox{ where }  A _i:= \bigcup _{j\in J_i} S_{i,j}, \mbox{ for }i<n.\label{EQ591}
\end{eqnarray}
Let $\X^{(1)}, \X_{i,j},\S_{i,j}, \A _i$ denote the corresponding suborders of $\X$.
By Claims 4.2--4.4 and 4.8 \cite{Kprod1} we have
\begin{fac}\label{T531}\rm
If $\X = \prod _{i<n}\X_i$, where $\X _i\in \CC ^{\rm rt}_{>1}$, for $i<n$, then we have
\begin{itemize}\itemsep -0.5mm
\item[\rm (a)] $X ^{(1)}= D_{\f _1,\X }$, where $\f _1 (v)$ is a formula saying that $v$ is not a minimal element and that the set $(\cdot ,v]$ is linearly ordered\footnote{Say,
\begin{eqnarray}
\f _1 (v) &:= & \exists u \;(u<v)\land  \forall u,w \;(u,w \leq v \Rightarrow u\leq w \lor w\leq u), \label{EQ943}\\[1.5mm]
\r (u,v)  &:= & \f _1(u) \land \f _1(v) \land \big(\neg \exists w \;\; (u\leq w \land v\leq w) \lor \exists w \;(\f _1 (w)\land w\leq u \land w\leq v) \big),\label{EQ935}\\
\ve_n     &:= &\textstyle \forall v_0,v_1,v_2 \;\big(\bigwedge _{i<3}\f _1 (v_i)\Rightarrow
                      \r (v_0,v_0) \land
                      (\r (v_0,v_1)\Rightarrow \r (v_1,v_0)) \land
                      (\r(v_0,v_1) \land \r(v_1,v_2) \Rightarrow \r(v_0 , v_2))\big)  \nonumber\\
          &   &\textstyle\land \;\exists v_0,\dots ,v_{n-1} \;\big(\bigwedge _{i<n}\f _1 (v_i)\land \bigwedge _{i<j<n} \neg \r(v_i, v_j)
                      \land \forall v_n \;(\f _1 (v_n)\Rightarrow \bigvee _{i<n}\r (v_n,v_i))\big).\label{EQ948}
\end{eqnarray}
\label{F000}};
\item[\rm (b)] $D _{\r , \X }= \bigcup _{i<n} A_i ^2 $
               (that is, $D _{\r , \X }$ is the equivalence relation on the set $X^{(1)}$ determined by the partition $\{ A_i:i<n\}$),
               where $\r(u,v)$ is a formula (see footnote \ref{F000}) saying that $u$ and $v$ are in $X^{(1)}$ and have no upper bound in $\X$ or have a lower bound  belonging to $X^{(1)}$;
\item[\rm (c)] $\S_{i,j} \cong  \X_{i,j} , \mbox{ for }j\in J_i, \;\mbox{ and }\;
               \A _i = \bcd _{j\in J_i} \S_{i,j} \cong \X _i ^+  \;\mbox{ and }\;
               \X^{(1)}=\bcd _{i<n}\A_i$;
\item[\rm (d)] For each $m \in [1,n]$ each $m$ elements of $X^{(1)}$ belonging to different sets $A_i$ have a supremum.
\end{itemize}
\end{fac}
\begin{cla}\label{T916}\rm
If $\X = \prod _{i<n}\X_i$, where $\X _i\in \CC ^{\rm rt}_{>1}$, $i<n$, then $\X \models \f ^{\rm po}_{0}\land \ve _n$,
where $\ve_n$ is a sentence (see footnote \ref{F000}) saying that $D _{\r , \cdot }$ is an equivalence relation on the set $D_{\f _1,\cdot }$ with exactly $n$ equivalence classes.
\end{cla}
\dok
It is evident that $\X \models \f ^{\rm po}_{0}$, while $\X \models \ve _n$ follows from Fact \ref{T531}(b).
\kdok
Further we show that for each $\bar x\in X$ there is a unique $n$-tuple of coordinates
$\la \bar x^0, \dots ,\bar x^{n-1}\ra\in \prod_{i<n} (\{ \bar 0\}\cup A_i)$ such that $\bar x=\bigvee _{i<n}\bar x^i$; i.e., $\bar x=\sup \{\bar x^i:i<n\}$.
(So, by Fact \ref{T531}(c) the suborders $\{ \bar 0\}\cup \A_i \cong \X _i$, $i<n$, of $\X$
can be regarded as ``coordinate trees" for the product $\X$.)
If $\bar x=\bar0$, then, clearly, $\la \bar x^0, \dots ,\bar x^{n-1}\ra=\la \bar 0,\dots,\bar 0\ra$ are such coordinates
and for $\bar x>\bar0$ we have the following lemma.
\begin{lem}\label{T550}
If $\X = \prod _{i<n}\X_i$, where $\X _i\in \CC ^{\rm rt}_{>1}$, for $i<n$, and if $m\in [1,n]$ and $\bar x\in X^{(m)}$, then
\begin{itemize}\itemsep -0.5mm
\item[\rm (a)] For each $i<n$ we have: $(\cdot ,\bar x]\cap A_i=\emptyset$ or $(\cdot,\bar x]\cap A_i$ is a linear order;\\
               also $(\cdot,\bar x]\cap (\{\bar 0\}\cup A_i)$ is the linear order $[\bar 0,\bar x ^i]$, where
               \begin{equation}\label{EQ950}
               \bar x^i:=\la 0_0,\dots ,0_{i-1},x_i,0_{i+1},\dots, 0_{n-1}\ra=\max ((\cdot,\bar x]\cap (\{\bar 0\}\cup A_i));
               \end{equation}
\item[\rm (b)] $(\cdot ,\bar x]= \prod _{i<n}[0_i, x _i]_{\X _i}\cong \prod _{i<n}[\bar 0,\bar x ^i]_\X$ is a distributive lattice,
               $\min (\cdot ,\bar x]=\bar 0$ and $\max (\cdot ,\bar x]=\bar x$;
\item[\rm (c)] If $K\in [n]^m$, where $K=\{ i<n : x_i>0_i\}$, then $\bar x=\bigvee _{i\in K}\bar x^i$, where $\bar x^i$-s are defined by (\ref{EQ950});
               defining $\bar x ^i=\bar 0$, for $i\in n\setminus K$ we have 
\begin{equation}\label{EQ960}\textstyle
\la \bar x^i:i<n\ra\in \prod_{i<n} (\{ \bar 0\}\cup A_i) \quad\mbox{ and }\quad\bar x=\bigvee _{i<n}\bar x^i ;
\end{equation}
\item[\rm (d)] If $\la \bar y^i:i<n\ra\in \prod_{i<n} (\{ \bar 0\}\cup A_i)$  and $\bar x=\bigvee _{i<n}\bar y^i$,
               then $\bar y^i=\bar x^i$, for each $i<n$.
\end{itemize}
\end{lem}
\dok
(a) If $(\cdot ,\bar x]\cap A_i=\emptyset$ we are done.
If $\bar y\in (\cdot,\bar x]\cap A_i$,
then by (\ref{EQ590}) and (\ref{EQ591}) we have $\bar y=\la 0_0,\dots ,0_{i-1},y,0_{i+1},\dots, 0_{n-1}\ra \in S_{i,j}$, for some $j\in J_i$,
thus $y\in X_{i,j}$ and by Fact \ref{T531}(c) $\S_{i,j}\cong \X_{i,j}$.
Since $\bar y \leq _\X \bar x=\la x_0,\dots ,x_{n-1}\ra$ we have $0_i<_{\X _i} y\leq _{\X _i}x_i$;
thus, $y$ and $x_i$ are in the same component of $\X _i^+$ and, hence, $y,x_i\in X_{i,j}$,
and $\bar x^i:=\la 0_0,\dots ,0_{i-1},x_i,0_{i+1},\dots, 0_{n-1}\ra \in S_{i,j}$,
which by (\ref{EQ591}) gives $\bar x^i\in (\cdot,\bar x]\cap A_i$.
We prove that $(\cdot,\bar x^i]\cap A_i =(\cdot,\bar x]\cap A_i$.
Since $\bar x^i\leq _\X\bar x$ the inclusion ``$\subset$" is true.
If $\bar y'\in (\cdot,\bar x]\cap A_i$,
then as for $\bar y$ we have $\bar y'=\la 0_0,\dots ,0_{i-1},y',0_{i+1},\dots, 0_{n-1}\ra $
and $y'\leq _{\X _i}x_i$,
which gives $\bar y'\in (\cdot,\bar x^i]$
and the inclusion ``$\supset$" is true as well.
Since $\bar x^i\in S_{i,j} \subset A_i \subset X^{(1)}$
by Fact \ref{T531}(a) the suborder $(\cdot,\bar x^i]\cap A_i=(\cdot,\bar x]\cap A_i$ of $\X$ is a linear order with a largest element $\bar x^i$.

Since $(\cdot,\bar x]\cap (\{\bar 0\}\cup A_i)=\{\bar 0\}\cup ((\cdot,\bar x]\cap A_i)=\{\bar 0\}\cup ((\cdot,\bar x ^i]\cap A_i)$
it remains to be proved that $\{\bar 0\}\cup ((\cdot,\bar x ^i]\cap A_i)=[\bar 0,\bar x ^i]$.
The inclusion ``$\subset$" is evident.
Conversely, if $\bar y\in (\bar 0,\bar x ^i]$,
then by (\ref{EQ950}) $\bar y=\la 0_0,\dots ,0_{i-1},y,0_{i+1},\dots, 0_{n-1}\ra $, where $0_i <_{\X _i} y\leq _{\X _i} x_i$.
So, $\bar x^i\in A_i$, $\bar y,\bar x^i \in X^{(1)}$ and $\bar y$ is a lower bound for $\{ \bar y,\bar x^i\}$,
which by Fact \ref{T531}(b) gives $\bar y\in A_i$.
Thus $\bar y\in (\cdot,\bar x ^i]\cap A_i$ indeed.

(b) Clearly, $\bar y\in (\cdot ,\bar x]_\X$
iff $0_i\leq _{\X_i} y_i \leq _{\X_i} x_i$, for each $i<n$,
iff $\bar y\in\prod _{i<n}[0_i, x _i]_{\X _i}$.
Since $\X _i$ is a tree its suborder $[0_i, x _i]_{\X _i}$ is a linear order and, hence, a distributive lattice.
So, $\prod _{i<n}[0_i, x _i]_{\X _i}$ is a distributive lattice too.
By (\ref{EQ950}) we have $[0_i, x _i]_{\X _i}\cong [\bar 0,\bar x ^i]_\X$, for $i<n$,
which gives $\prod _{i<n}[0_i, x _i]_{\X _i}\cong \prod _{i<n}[\bar 0,\bar x ^i]_\X$.

(c) Since $\bar x^i\in A_i$, for $i\in K$, by Fact \ref{T531}(d) $\bar y:=\bigvee _{i\in K}\bar x^i$ exists.
Since $\bar x ^i \leq \bar x$, for $i<n$, we have $\bar y\leq \bar x$.
For $i\in K$ we have $\bar x ^i \leq \bar y\leq \bar x$
and, by (\ref{EQ950}), $x_i \leq y_i \leq x_i$; so, $y_i=x_i$.
For $i\in n\setminus K$ we have $y_i\leq x_i =0_i$ and, hence, $y_i= x_i =0_i$ again.
So, $\bar x =\bar y=\bigvee _{i\in K}\bar x^i$.
The rest is evident.

(d) Let $\la \bar y^i:i<n\ra$ be a assumed.
Then $\bar y^i,\bar x^i\in (\cdot ,\bar x]_\X$, for $i<n$,
and by (b) we can work in the lattice $(\cdot ,\bar x]_\X$.
For $i\neq j$ we have $\bar x^i\in \{ \bar 0\}\cup A_i$ and  $\bar y^j\in \{ \bar 0\}\cup A_j$;
so, assuming that there is $\bar z$ such that $\bar 0 <\bar z \leq \bar x^i,\bar y^j$
we would have $\X\models \r [\bar x^i,\bar y^j]$
and by Fact \ref{T531}(b) $i=j$, which is false.
Thus $\bar x^i \land \bar y^j=\bar 0$, for $i\neq j$.
Now, for $j<n$ we have $\bar y^j =\bar y^j \land \bar x= \bar y^j \land \bigvee _{i<n}\bar x^i = \bigvee _{i<n}\bar y^j \land \bar x^i =\bar y^j \land \bar x^j$;
thus $\bar y^j \leq _\X \bar x^j$
and, symmetrically, $\bar x^j \leq _\X \bar y^j$,
which gives $\bar y^j = \bar x^j$, for all $j<n$.
\kdok
The formula $\a _n( v_0,\dots ,v_{n-1},v_n):= \bigwedge _{k < n}(v_k\leq v_n) \land \forall u \;(\bigwedge _{k< n}(v_k\leq u) \Rightarrow v_n \leq u)$
saying that (in partial orders) $v_n=\sup\{v_0,\dots ,v_{n-1}\}$ will be used in the sequel.
\begin{cla}\label{T917}\rm
If $\X = \prod _{i<n}\X_i$, where $\X _i\in \CC ^{\rm rt}_{>1}$, for $i<n$, then $\X \models \nu \land \mu\land \l _n$, where \\[-6mm]
\begin{itemize}\itemsep -1mm
\item $\nu$ is a sentence saying that each two $\r$-equivalent elements of $D_{\f _1,\cdot}$ having an upper bound
      are comparable\footnote{Instead of $v_0,\dots ,v_{n-1}$ and $w_0,\dots ,w_{n-1}$ we write $\bar v$ and $\bar w$\label{F001}
      \begin{eqnarray}
      \nu  & := & \textstyle\forall v,v_0,v_1 \big(\f _1 (v_0) \land \f _1 (v_1)\land \r(v_0,v_1)\land v_0\leq v\land v_1\leq v)\Rightarrow v_0\leq v_1 \lor v_1\leq v_0\big)\label{EQ945}\\
      \mu  & := & \textstyle\forall \bar v \big(\bigwedge _{i<n}(v_i=0\lor\f _1 (v_i))\land
                  \bigwedge _{i<j<n} (v_i,v_j>0 \Rightarrow \neg \r(v_i, v_j))\Rightarrow \exists v_n \;\a _n( \bar v,v_n)\big)\label{EQ946}\\
      \l_n & := & \textstyle\forall v\; \exists \bar v  \;\big[
                                           \bigwedge_{i<n}(v_i=0 \lor \f _1(v_i))\land
                                           \bigwedge_{i<j<n}(v_i,v_j >0 \Rightarrow \neg \r (v_i,v_j)) \land
                                           \a_n (\bar v,v)) \land \nonumber\\
     &    & \textstyle\forall \bar w \big(
                                           \bigwedge_{i<n}(w_i=0 \lor \f _1(w_i))\land
                                           \bigwedge_{i<j<n}(w_i, w_j>0  \Rightarrow \neg \r (w_i,w_j) ) \land
                                           \a_n (\bar w ,v))\label{EQ949}\\
     &    &\textstyle \Rightarrow \bigvee _{\pi\in \Sym (n)}\bigwedge_{i<n} w_i=v_{\pi (i)}\big)\big]\nonumber.
      \end{eqnarray}
      };
or equivalently, saying that for each  point $v$ and each $\r$-equivalence class $C$ the set $(\cdot, v]\cap C$ is empty or linearly ordered;
\item $\mu$ is a sentence (see footnote \ref{F001}) saying that a supremum exists for each $n$-tuple  of elements of $\{0\}\cup D_{\f _1,\cdot}$
such that different non-zero entries belong to different $\r$-equivalence classes;
\item $\l_n$ is a sentence (see footnote \ref{F001}) saying that each point is a supremum of an $n$-tuple  of elements of $\{0\}\cup D_{\f _1,\cdot}$
such that different non-zero entries belong to different  $\r$-equivalence classes
and that such $n$-tuple is unique up to a permutation $\pi \in \Sym (n)$.
\end{itemize}
\end{cla}
\dok
Recall that by Fact \ref{T531}(a) we have $X^{(1)}=D_{\f _1,\X}$,
by Fact \ref{T531}(c) $\X^{(1)}=\bcd _{i<n}\A_i$, where
\begin{equation}\label{EQ939}
A_i:=\{0_0\}\times \dots \times \{0_{i-1}\}\times X _i ^+ \times \{0_{i+1}\}\times \dots \times \{0_{n-1}\},
\end{equation}
and by Fact \ref{T531}(b) $D _{\r , \X }$ is the equivalence relation on $X^{(1)}$ corresponding to the partition $\{ A_i:i<n\}$.

$\X \models \nu$. By (\ref{EQ945}) we have $\X \models \nu$ iff for each $\bar x\in X$ and each $i<n$ we have: $(\cdot,\bar x]\cap A_i$ is a linear order or $\emptyset$.
By  Lemma \ref{T550}(a), this is true for $\bar x>\bar 0$
and if $\bar x=\bar 0$, then $(\cdot,\bar x]\cap A_i=\emptyset$;
thus, $\X \models \nu$.

$\X \models \mu$.
By (\ref{EQ946}) we have $\X\models \mu$ iff $\bigvee_{i<n}\bar x^i$ exists for each $\la \bar x^0,\dots ,\bar x ^{n-1}\ra\in \prod_{i<n}(\{\bar 0_\X\} \cup A_i)$.
If $K:=\{ i<n:\bar x^i>\bar 0\}$, then by Fact \ref{T531}(d) $\bigvee_{i\in K}\bar x^i$ exists
and, clearly, $\bigvee_{i<n}\bar x^i=\bigvee_{i\in K}\bar x^i$.

$\X \models \l_n$. By (\ref{EQ949}), $\X\models \l _n$ iff for each $\bar x\in X$ there are $\bar x^i\in \{\bar 0\}\cup X^{(1)}$, for $i<n$, such that:
\begin{itemize}\itemsep -1mm
\item[\sc (i)]   If $i\neq j$, and $\bar x^i, \bar x^j>\bar 0$, then $\bar x^i$ and $\bar x^j$ belong to different $\r$-equivalence classes, $A_i$, $i<n$,
\item[\sc (ii)]  $\bar x=\bigvee _{i<n}\bar x^i$,
\item[\sc (iii)] If $\bar y^i\in \{\bar 0\}\cup X^{(1)}$, for $i<n$, and \\[-7mm]
                 \begin{itemize}\itemsep -1mm
                 \item[(i)] If $i\neq i'$ and $\bar y^i,\bar y^{i'}>\bar 0$, then $\bar y^i \in A_{j_i}$ and $\bar y^{i'}\in A_{j_{i'}}$, for different $j_i,j_{i'}<n$,
                 \item[(ii)] $\bar x=\bigvee _{i<n}\bar y^i$,\\[-7mm]
                 \end{itemize}
                 then there is $\pi\in \Sym (n)$ such that $\bar y^i= \bar x^{\pi (i)}$, for $i<n$.
\end{itemize}
So, if $\bar x=\la x_0, \dots ,x_{n-1}\ra\in X=\prod _{i<n}X_i$,
then by Lemma \ref{T550}(a)
\begin{equation}\label{EQ941}
\bar x^i:=\la 0_0,\dots ,0_{i-1},x_i,0_{i+1},\dots, 0_{n-1}\ra \in \{\bar 0\}\cup A_i,\quad\mbox{ for } i<n,
\end{equation}
and by Lemma \ref{T550}(c) $\bar x=\bigvee _{i<n}\bar x^i$; thus {\sc (i)} and {\sc (ii)}  are true.
Clearly we have
\begin{equation}\label{EQ937}
J:=\{ i<n :\bar x^i>\bar 0 \}=\{ i<n :x_i>0_i \}.
\end{equation}
In order to prove {\sc (iii)}, we assume that $\bar y^i\in \{0\}\cup X^{(1)}$, for $i<n$, and that (i) and (ii) are true.

If $\bar x=\bar 0$,
then, since $\bar x=\bigvee _{i<n}\bar x^i=\bigvee _{i<n}\bar y^i$
we have $\bar x^i=\bar y^i =\bar 0$, for all $i<n$, and {\sc (iii)} is true.

Otherwise we have $J\neq \emptyset$ and
\begin{equation}\label{EQ940}
K:= \{ i<n: \bar y^i>\bar 0\}\neq \emptyset.
\end{equation}
For $i\in K$ we have $\bar y^i\in X^{(1)}=\bcd _{j<n}A_j$;
thus, there is a unique $j_i$ such that $\bar y^i\in A_{j_i}$.
So, by (\ref{EQ939}),
\begin{equation}\label{EQ938}
\bar y^i=\la  0_0, \dots, 0_{j_i-1}, y^i_{j_i},0_{j_i+1}, \dots ,0_{n-1} \ra\in A_{j_i}, \quad \mbox{ for } i\in K,
\end{equation}
and we prove that $J=\{ j_i:i\in K\}$.
If $i\in K$, then by (\ref{EQ938}) we have $y^i_{j_i}>0_{j_i}$,
and, since by (ii) $\bar x\geq _\X \bar y^i$, we have $x_{j_i}>0_{j_i}$,
which by (\ref{EQ937}) gives $j_i\in J$.
On the other hand, if $j\in J$, then by (\ref{EQ937}) $x_j>0_j$
and, by (\ref{EQ938}) and Lemma \ref{T550}, $j=j_i$, for some $i\in K$.

Thus $J=\{ j_i:i\in K\}$,
by (i) for different $i,i'\in K$ we have $j_i \neq j_{i'}$,
and, hence, the mapping $p:K\rightarrow J$ defined by $p(i)=j_i$, for $i\in K$, is a bijection.
Let $\pi \in \Sym (n)$ be a permutation which extends $p$
and for $i<n$ let us define $\bar z^{\pi (i)}:=\bar y ^{i}$.
Then for $i\in K$  by (\ref{EQ938}) we have $\bar z^{\pi (i)} \in A_{j_i}=A_{p(i)}=A_{\pi(i)}\subset \{ \bar 0\}\cup A_{\pi(i)}$,
while for $i\in n\setminus K$ by (\ref{EQ940}) we have $\bar z^{\pi (i)}=\bar 0 \in \{ \bar 0\}\cup A_{\pi(i)}$ again.
Thus, $\bar z^{\pi (i)}\in \{ \bar 0\}\cup A_{\pi(i)}$, for all $i<n$,
and, since $\pi \in \Sym (n)$, $\bar z^i\in \{ \bar 0\}\cup A_i$, for all $i<n$.
Since $\{\bar z^i:i<n \}=\{\bar y^i:i<n \}$,
by (ii) we have $\bar x=\bigvee _{i<n}\bar z^i$.
Since by {\sc (ii)} $\bar x=\bigvee _{i<n}\bar x^i$ too and $\bar x^i,\bar z^i\in \{ \bar 0\}\cup A_i$, for all $i<n$,
by Lemma \ref{T550}(d) for each $i<n$ we have $\bar x^i=\bar z^i$,
and, hence, $\bar y ^{i}=\bar z^{\pi (i)}=\bar x^{\pi (i)}$.
So {\sc (iii)} is true indeed and $\X \models \l _n$.
\kdok
Let
\begin{equation}\label{EQ959}
\t_{\CC _n} :=\f ^{\rm po}_{0} \land \ve_n \land \nu \land \mu \land \l _n, \;\;\mbox{ for }n\geq 2.
\end{equation}
\begin{te}\label{T906}
If $n\geq 2$ and $\X \in \CC _n$, then $\X\models \t_{\CC _n}$.
\end{te}
\dok
Let $\X \cong \prod _{i<n}\X_i$,
where $\X _i\in \CC ^{\rm rt}_{>1}$, for $i<n$.
Since $\X\models \t_{\CC _n}$ iff $\prod _{i<n}\X_i\models \t_{\CC _n}$ w.l.o.g.\ we assume that $\X =\prod _{i<n}\X_i$.
By Claim \ref{T916} we have $\X\models\f ^{\rm po}_{0} \land \ve_n$
and, by Claim \ref{T916}, $\X\models \nu \land \mu \land \l _n$.
\hfill $\Box$
\subsection{Proof of the implication ``$\Leftarrow$" in (\ref{EQ958})}
\begin{te}\label{T908}
If $\X$ is an $L_b$-structure and $\X\models \t_{\CC _n}$, then $\X \in \CC _n$.
\end{te}
\dok
Let $\X\models \t_{\CC _n}$.
Since $\X\models \f ^{\rm po}_{0}$ the structure $\X$ is a partial order a smallest element;
so we can write $\X =\la X,\leq \ra$ and define $0:=\min \X$.
By (\ref{EQ943}) we have
\begin{equation}\label{EQ944}\textstyle
X^{(1)}:=D_{\f _1,\X}=\{ x\in X\setminus \{ 0\}: \mbox{ the set $(\cdot ,x]$ is linearly ordered by}\leq \}.
\end{equation}
\begin{cla}\label{T909}\rm
Under the assumptions we have
\begin{itemize}\itemsep -0.3mm
\item[\rm (i)] $D_{\r ,\X}=\r$, where $\r$ is the binary relation on the set $X^{(1)}$ defined by
             \begin{equation}\label{EQ930}\textstyle
             x\rr y \Leftrightarrow (\exists z\in X^{(1)} \;\; z\leq x,y)\lor (\neg \exists z\in X   \;\; z\geq x,y);
             \end{equation}
             $\r$ is an equivalence relation on $X^{(1)}$ with $n$ equivalence classes, say $X^{(1)}\!/\,\r=\{ A_0,\dots ,A_{n-1}\}$;
\item[\rm (ii)]  $(\cdot,x]\cap (\{0\}\cup A_i)$ is a linear order, for each $x\in X$ and each $i<n$,
\item[\rm (iii)] If $\bar x\in \prod_{i<n}(\{0\} \cup A_i)$, then $\bigvee_{i<n}x_i$ exists,
\item[\rm (iv)] For each $x\in X$ there is $\bar x\in \prod_{i<n}(\{0\} \cup A_i)$ such that $x=\bigvee_{i<n}x_i$
             and such $n$-tuple is unique; that is,
             \begin{equation}\label{EQ933}\textstyle
             \forall \bar y\in \prod_{i<n}(\{0\} \cup A_i) \;\;
             (x=\bigvee_{i<n}y_i\;\Rightarrow\;\bar y=\bar x).
             \end{equation}
\end{itemize}
\end{cla}
\dok
(i) By (\ref{EQ935}) and (\ref{EQ944}) for $x,y\in X$ we have $\la x,y\ra \in D_{\r ,\X}$
iff $\X \models \r [x,y]$,
iff $x,y\in X^{(1)}$
and $z\leq x,y$, for some $z\in X^{(1)}$, or the set $\{ x,y\}$ has no upper bound in $\X$.
Thus $\r =D_{\r ,\X}$, where $\r$ is the binary relation on the set $X^{(1)}$ defined by (\ref{EQ930}).
Since $\X\models \ve_n$, for each $x,y,z\in X^{(1)}$ we have
$\la x,x\ra\in \r$,
$\la x,y\ra\in \r $ implies $\la y,x\ra\in \r $, and
$\la x,y\ra,\la y,z\ra\in \r $ implies $\la x,z\ra\in \r $;
thus $\r$ is an equivalence relation on the set $X^{(1)}$.
In addition, there are $x_0,\dots ,x_{n-1}\in X^{(1)}$ belonging to different $\r$-equivalence classes,
say $x_i\in A_i$, for $i<n$, and $\bigcup _{i<n}A_i=X^{(1)}$;
thus, $X^{(1)}\!/\,\r=\{ A_0,\dots ,A_{n-1}\}$ and we have (i). (Here an enumeration of the quotient $X^{(1)}\!/\,\r$ is fixed.)

(ii) Since $\X\models \nu $,
by (\ref{EQ945}), (\ref{EQ944}) and (i) for each $x\in X$ and each $y,z\in X^{(1)}$ such that $y,z\leq x$ and $y\rr z$ we have $y\leq z$ or $z\leq y$.
So, if $x\in X$,  $i<n$ and  $y,z\in (\cdot ,x]\cap A_i$, then $y$ and $z$ are comparable,
which means that $(\cdot ,x]\cap A_i$ is a linear order.
Clearly, $(\cdot ,x]\cap (\{ 0\} \cup A_i)$ is a linear order too.

(iii) Since $\X\models \mu $, by (\ref{EQ946}), (\ref{EQ944}) and (i) we have:
if $x_i \in \{ 0\}\cup X^{(1)}$, for $i<n$, $K:=\{ i<n : x_i>0\}$ and $\la x_i,x_j\ra\not\in \r$, for different $i,j\in K$,
then $\bigvee _{i<n}x_i$ exists.
Now, if $\bar x\in \prod_{i<n}(\{0\} \cup A_i)$, then $x_i \in \{ 0\}\cup A_i \subset \{ 0\}\cup X^{(1)}$, for $i<n$,
and $K:=\{ i<n : x_i>0\}=\{ i<n: x_i \in A_i\}$;
so, for different $i,j\in K$ we have $x_i \in A_i \neq A_j \ni x_j$
and, by (i), $\la x_i,x_j\ra\not\in \r$;
consequently, $\bigvee_{i<n}x_i$ exists and (iii) is true.

(iv) Since $\X\models\l _n$, by (\ref{EQ949}), (\ref{EQ944}) and (i) we have: for each $x\in X$ there are $a_k\in \{0\}\cup X^{(1)}$, for $k<n$, such that\\[-7mm]
\begin{itemize}\itemsep -1mm
\item[1.] $\la a_k,a_{k'}\ra\not\in \r$, for different $k,k'\in K:=\{ k<n : a_k>0\}$,
\item[2.] $x=\bigvee _{k<n}a_k$,
\item[3.] If $y_i\in \{0\}\cup X^{(1)}$, for $i<n$, and\\[-7mm]
         \begin{itemize}\itemsep -0.1mm
         \item[3.1] $\la y_i,y_{i'}\ra\not\in \r$, for different $i,i'\in L:=\{ i<n : y_i>0\}$,
         \item[3.2] $x=\bigvee _{i<n}y_i$,\\[-7mm]
         \end{itemize}
         then there is $\pi\in \Sym (n)$ such that $y_i= a_{\pi (i)}$, for $i<n$.\\[-7mm]
\end{itemize}
Now, if $x\in X$, then there are $a_k\in \{0\}\cup X^{(1)}$, for $k<n$, satisfying 1--3.
By (i) for $k\in K$ there is $i_k<n$ such that $a_k\in A_{i_k}$;
and by 1 for different $k,k'\in K$ we have $i_k \neq i_{k'}$.
Thus the function $I_{\bar a}:K \rightarrow n$ defined by $I_{\bar a}(k)= i_k$, for $k\in K$, is a partial injection from $n$ to $n$.
Let $\pi _{\bar a}\in \Sym (n)$, where $I_{\bar a}\subset \pi_{\bar a}$ and let us define
\begin{equation}\label{EQ947}\textstyle
x_i :=a_{\pi _{\bar a}^{-1}(i)}, \quad \mbox{ for }i<n.
\end{equation}
Then for $i\in I_{\bar a}[K]$ we have $i=I_{\bar a}(k)=\pi_{\bar a}(k)$, for some $k\in K$,
and, hence, $x_i =a_{\pi _{\bar a}^{-1}(i)}=a_{\pi _{\bar a}^{-1}(\pi_{\bar a}(k))}=a_k\in A_{i_k}=A_{I_{\bar a}(k)}=A_i $.
On the other hand, for $i\in n\setminus I_{\bar a}[K]=\pi_{\bar a}[n\setminus K]$
we have $i=\pi_{\bar a}(j)$, for some $j\in n\setminus K$,
and, hence, $x_i =a_{\pi _{\bar a}^{-1}(i)}=a_{\pi _{\bar a}^{-1}(\pi_{\bar a}(j))}=a_j =0$.
Thus $x_i \in A_i $, for $i\in I_{\bar a}[K]$, and $x_i=0$, for $i\in n\setminus I_{\bar a}[K]$,
which gives $\bar x:=\la x_0,\dots ,x_{n-1}\ra \in \prod _{i<n}(\{ 0\}\cup A_i)$.
By (\ref{EQ947}) we have $\{ x_i:i<n\}=\{ a_i:i<n\}$ which by 2 implies $x=\bigvee _{i<n}x_i$.

Let $\bar y:=\la y_0,\dots ,y_{n-1}\ra \in \prod _{i<n}(\{ 0\}\cup A_i)$,
let $x=\bigvee _{i<n}y_i$ and $L:=\{ i<n : y_i>0\}=\{ i<n : y_i\in A_i\}$.
Then $y_i\in \{0\}\cup X^{(1)}$, for $i<n$,
and for different $i,i'\in L$ we have $y_i\in A_i$ and $y_{i'}\in A_{i'}$,
which by (i) gives $\la y_i,y_{i'}\ra\not\in \r$.
So, by 3, there is $\pi\in \Sym (n)$ such that $y_i= a_{\pi (i)}$, for $i<n$,
and, hence, $\{ y_i:i<n\}=\{ a_i:i<n\}$.
Since $\{ x_i:i<n\}=\{ a_i:i<n\}$ too we have $\{ y_i:i<n\}=\{ x_i:i<n\}$.
Now, for $j\in L$ we have $A_j\ni y_j \in \{ x_i:i<n\}$ and, hence $y_j=x_j$.
If $j\in n\setminus L$, then $y_j=0$;
assuming that $x_j \in A_j$ we would have $x_j =y_k$, for some $k<n$,
but then $k=j$ and $y_j \in A_j$, which gives a contradiction.
Thus $x_j=0=y_j$ again, and $\bar y=\bar x$ indeed.
\hfill $\Box$
\begin{cla}\label{T910}
$\T _i:=\{0\} \cup \A_i$, $i<n$, are rooted trees of size $>1$ and $\X \cong \prod_{i<n}\T _i$.
\end{cla}
\dok
Precisely, $\T _i=\la \{0\} \cup A_i, \leq _i\ra$, where $\leq _i\;=\;\leq\upharpoonright (\{0\} \cup A_i)$, for $i<n$.
For $i<n$ by (i) we have $\emptyset \neq A_i\subset X^{(1)}$ and we first show that $\A _i$ is a tree.
For $x\in A_i$ and $y\in (\cdot ,x]_{\A _i}$ we have $y\in (\cdot ,x]_\X$;
thus $(\cdot ,x]_{\A _i}\subset (\cdot ,x]_\X$;
since $x\in X^{(1)}$ the set $(\cdot ,x]_\X$ is linearly ordered
and, hence, $(\cdot ,x]_{\A _i}$ is linearly ordered too.
So $\A _i$ is a tree and, consequently,
$\T _i =\{0\} \cup \A _i$ is a rooted tree, $\min \T _i =0$ and $|T_i|>1$.

By (iii) the mapping $f:\prod_{i<n}(\{0\} \cup A_i) \rightarrow X$ given by
$$\textstyle
f(\la x_0, \dots x_{n-1} \ra)=\bigvee_{i<n}x_i , \quad \mbox{ for all }\bar x=\la x_0, \dots x_{n-1} \ra\in \prod_{i<n}(\{0\} \cup A_i),
$$
is well defined
and we show that $f: \prod_{i<n}\T _i\rightarrow \X$ is an isomorphism.
First, if $\bar x, \bar y\in \prod_{i<n}(\{0\} \cup A_i)$ and $f(\bar x)=f(\bar y)$,
that is, $\bigvee_{i<n}x_i=\bigvee_{i<n}y_i$,
then by (iv) we have $\bar x=\bar y$; thus, $f$ is an injection.
Second, if $x\in X$,
then by (iv) there is $\bar x\in \prod_{i<n}(\{0\} \cup A_i) $ such that $x=\bigvee_{i<n}x_i=f(\bar x)$;
so, $f$ is a surjection.

Third, for $\bar x, \bar y\in \prod_{i<n}(\{0\} \cup A_i)$ we prove that
\begin{equation}\label{EQ934}\textstyle
\bar x \leq \bar y \Leftrightarrow \bigvee _{i<n} x_i \leq \bigvee _{i<n} y_i.
\end{equation}
If $\bar x \leq \bar y$,
then for each $i<n$ we have $x_i \leq  y_i \leq \bigvee _{i<n} y_i$,
which implies that $\bigvee _{i<n} x_i \leq \bigvee _{i<n} y_i$
and the implication ``$\Rightarrow$" in (\ref{EQ934}) is proved.

Conversely, assuming that $x:=\bigvee _{i<n} x_i \leq \bigvee _{i<n} y_i =:y$,
we prove that $\bar x \leq \bar y$.
If $x=y$, then by (iv) we have $\bar x = \bar y$ and we are done.
Otherwise we have $x<y$.
Suppose that $\bar x \not\leq  \bar y$.
Then there is $i_0 <n$ such that $x_{i_0}\not\leq  y_{i_0}$
and, since $x_{i_0}, y_{i_0}\in (\cdot ,y]\cap (\{0\} \cup A_{i_0})$,
by (ii) we have $x_{i_0}> y_{i_0}$.
Now $\bar y<\bar z:=\la y_0,\dots , y_{i_0-1},x_{i_0}, y_{i_0-1},\dots ,y_{n-1} \ra\in \prod_{i<n}(\{0\} \cup A_i)$,
by (iii) $\bigvee_{i<n}z_i$ exists,
since $\bar y<\bar z$ we have $y=\bigvee_{i<n}y_i\leq \bigvee_{i<n}z_i$
and since $z_i\leq y$, for all $i<n$, we have $\bigvee_{i<n}z_i\leq y$.
Thus $\bigvee_{i<n}z_i=\bigvee_{i<n}y_i= y$ and $\bar y\neq \bar z$,
which is false by (iv).
So $\bar x \leq  \bar y$ and (\ref{EQ934}) is true.
\hfill $\Box$
\section{Vaught's conjecture for the partial orders from $\la\la \CC ^{\rm rt}\ra_{\Pi} \ra _{\Sigma}$}\label{S4}
Recall that for $n\geq 2$ the class $\CC _n :=\{ \Y \in \Mod _{L_b}:\exists \X_0,\dots,\X_{n-1}\in \CC ^{\rm rt}_{>1}\; \Y \cong \prod _{i<n}\X_i \}$
is defined by the sentence $\t_{\CC _n}$ given by (\ref{EQ959}).
For $n=1$ let $\CC _1:=\CC ^{\rm rt}_{>1}$
and let $\t _{\CC _1}$ be a first-order $L_b$-sentence defining $\CC _1$.
So, for each $n\in \N$ and for each $L_b$-structure $\X$ we have
\begin{equation}\label{EQ942}
\X \mbox{ is isomorphic to a direct product of $n$ rooted trees of size $>1$}\;\;\Leftrightarrow\;\; \X\models \t_{\CC _n} .
\end{equation}
\begin{te}\label{T907}
If $\Y \in\la\la \CC ^{\rm rt}\ra_{\Pi} \ra _{\Sigma}$
and $\Y \cong \X=\sum _{\BI}\prod _{j<m_i}\X_{i,j}$, where $\BI\in \CC ^{\rm fin}$,
$m_i \in \N$ and $\X_{i,j}\in \CC ^{\rm rt}$, for $j<m_i$ and $i\in I$,
then

(a) The theory $\CT :=\Th(\X)=\Th (\Y)$ satisfies Vaught's conjecture and (\ref{EQ754}) holds;

(b) $\CT$ is $\o$-categorical iff $\Th (\X_{i,j})$ is $\o$-categorical, for each $j<m_i$ and $i\in I$;

(c) If all the factors $\X_{i,j}$ satisfy VC$^\sharp$, then $I(\CT)\in \{ 1,\c\}$.
\end{te}
\dok
(a) For each $i\in I$ the poset $\X _i:=\prod _{j<m_i}\X_{i,j}$ has a smallest element
and, hence, $\CT$ is an FLD$_0$-theory and $\X=\sum _{\BI}\X _i\in \CD (\CT)$.
We show that the decomposition $\sum _{\BI}\X _i$ witnesses that the theory $\CT$ is actually Vaught's.
For $i\in I$ it is possible that $|X_i|=1$;
then there is a sentence $\t_i\in \Th (\X_i)$ saying that and, clearly, $\t _i$ provides VC (in the sense of (\ref{EQ741})).
Otherwise we have $|X_i|>1$
and, excluding factors $\X_{i,j}$ of size 1,
w.l.o.g.\ we can assume that $m_i\in \N$ is taken such that $|X_{i,j}|>1$, for all $j<m_i$.
So, $\X _i$ is a product of $m_i$ rooted trees of size $>1$
and by (\ref{EQ942}) $ \t _{\CC_{m_i}}\in \Th (\X_i )$.
If $\Z \models \t _{\CC_{m_i}}$,
then, by (\ref{EQ942}) again, $\Z $ is isomorphic to a direct product of $n$ rooted trees of size $>1$,
and by Fact \ref{T544}(b) we have $I(\Z)\leq \o$ or $I(\Z)=\c$;
so the sentence $\t _{\CC_{m_i}}$ provides VC (in the sense of (\ref{EQ741})).
Thus the theory $\CT$ is actually Vaught's and, by Fact \ref{T712}, VC is true for $\CT$.

(b) By the dual of Theorem 3.3 of \cite{KFLD} $\CT$ is $\o$-categorical
iff $\prod _{j\in J} I(\Y _j)=1$, for some $\sum _{\BJ}\Y_j\in \CD (\CT)$,
iff $\prod _{j\in J} I(\Y _j)=1$, for all $\sum _{\BJ}\Y_j\in \CD (\CT)$.
So we can consider only the given decomposition $\sum _{\BI}\X_i$.
Thus $\CT$ is $\o$-categorical iff $I(\X_i)=1$, for all $i\in I$,
iff (by Fact \ref{T544}(b)) $I(\X_{i,j})=1$, for all $j<m_i$ and $i\in I$.

(c) If $I(\X_{i,j})\in \{1,\c\}$, for all $j<m_i$ and $i\in I$,
then by Fact \ref{T544}(b) $I(\X_i)\in \{1,\c\}$, for all $i\in I$,
so $\prod_{i\in I}I(\X_i)\in \{1,\c\}$ and $I(\CT)\in \{ 1,\c\}$ by (\ref{EQ754}).
\kdok
By the previous theorem the cardinal $I(\CT )$ satisfies (\ref{EQ754})
and in the ``otherwise" case we only have  the information that $I(\CT)\in[3,\o]$.
A negative answer to the following question would imply that in the ``otherwise" case we have $I(\CT)=\o$,
which would strengthen Fact \ref{T712} and its consequences.
\begin{que}\label{Q000}
Is there an Ehrenfeucht theory of partial order (in the language $\la \leq\ra$)?
\end{que}
\begin{rem}\label{R901}\rm
If Ehrenfeucht theories of partial order exist,
then in the ``otherwise" case of Theorem \ref{T907} $\CT$ is small, $I(\CT)\in [3,\o]$
and  $3\leq \prod _{i\in I} I(\Y _i)<\c$, for each decomposition $\sum _{\BI}\Y_i\in \CD (\CT)$.
Assuming that $\prod _{i\in I}I(\Y _i)>\o$, for some $\sum _{\BI}\Y_i\in \CD (\CT)$,
by Theorem 3.4 of \cite{KFLD} we would have $I(\CT )\geq \o _1$, which is false. Thus
$$\textstyle
\forall \sum _{\BJ}\Y_j\in \CD (\CT)\;\; \prod _{j\in J} I(\Y _j)\in [3,\o]
$$
and, in particular, $\prod _{i\in I} I(\X _i)\in [3,\o]$.
Let $\X ^{\rm at}\models \CT$ be a countable atomic model. It is possible that
\begin{itemize}
\item[1.]  $\exists \sum _{\BJ}\Y_j\in \CD (\X ^{\rm at})\;\;\prod_{j\in J} I(\Y _j)\in [3,\o)$.
Then by Theorem 3.5 of \cite{KFLD} $I(\CT) \in [3,\o)$; so $\CT$ is an Ehrenfeucht theory.
\item[2.]  $\forall \sum _{\BJ}\Y_j\in \CD (\X ^{\rm at})\;\;\prod_{j\in J} I(\Y _j)=\o$.
Is it true that then $I(\CT)=\o$ (or $I(\CT)<\o$ is possible)?
\end{itemize}
We note that if in Theorem \ref{T907} the factors $\X_{i,j}$ satisfy VC$^\sharp$, then the ``otherwise" case does not appear.
\end{rem}
\begin{ex}\label{EX102}\rm
By \cite{KFMD} (see also \cite{KMon,KAC}) VC$^\sharp$ is true for all theories admitting finite monomorphic decompositions, FMD theories;
i.e.\ the theories of relational structures definable in linear orders colored into finitely many convex colors
by quantifier free formulas.
In particular, by Proposition 5.2 of \cite{Ksharp},
a partial order $\X$ is rooted  FMD tree iff $\X=\sum _{\BK}\X _k$, where $\BK$ is a finite rooted tree,
and if $\M (\BK)$ is the set of maximal elements of $\BK$ and $k_0=\min \BK$, then
$\X _k$ is a chain, for $k\in K\setminus \M (\BK)$; $\X _k$ is a chain or an antichain, for $k\in \M (\BK )$;
and $\min \X _{k_0}$ exists.
Then by Theorem 5.4 of \cite{Ksharp} $I(\X)=1$ iff $I(\X _k)=1$, for all $k\in K$; otherwise, $I (\X )=\c$.
(By Fact \ref{T946}(b), if $I(\X)=1$, then  $\Th(\X)$ is finitely axiomatizable iff $\X_k$ is finite, for each $k\in K$ such that $\X _k$ is an antichain.)

So, by Theorem \ref{T907},
if $\Y\cong\sum _{\BI }\prod _{j<m_i}\X _{i,j}$,
where $\BI\in \CC ^{\rm fin}$ and all the factors $\X_{i,j}=\sum _{\BK_{i,j}}\X _{i,j,\,k}$ are rooted FMD trees,
then $I(\Y)=1$ iff $I(\X _{i,j,k})=1$, for all indices $i$, $j$ and $k$; and, otherwise, $I (\Y )=\c$.
We note that if, in addition, the summands $\X _{i,j,\,k}$ are finitely axiomatizable linear orders with a smallest element,
then by Theorem \ref{T940} the theory $\Th (\Y)$ is finitely axiomatizable and we have a simple characterization of its models:
$\Y '\equiv \Y$ iff $\Y ' \cong\sum _{\BI }\prod _{j<m_i}\sum _{\BK_{i,j}}\X _{i,j,\,k}'$, where $\X _{i,j,\,k}'\equiv \X _{i,j,\,k}$, for all indices $i$, $j$ and $k$.
\end{ex}
\section{First-order definable subclasses of $\CC _n$}\label{S5}
By Theorem \ref{T915} the class $\CC _n$ consisting of partial orders $\Y$ isomorphic to a product $\prod _{i<n}\X _i$ of $n$ rooted trees of size $>1$ is first-order definable.
Here we extend that result showing that, roughly, each subclass $\CC$ of $\CC _n$
isolated by a set of first-order properties of the factors $\X_i$ is first-order definable again.
\begin{te}\label{T939}
Whenever $k\leq n \in \N$ and $\t _0, \dots,\t _{k-1}\in \Sent _{L_b}$ there is a $L_b$-sentence defining the class
\begin{equation}\label{EQ979}\textstyle
\CC_n ^{\t _0, \dots,\t _{k-1}}:=\{ \Y\in \CC _n: \exists \X _0,\dots,\X _{n-1}\in \CC ^{\rm rt}_{>1}\;(\Y \cong \prod _{i<n}\X _i \land \forall i<k \;\X _i\models \t _i)\}.
\end{equation}
\end{te}
A proof of the theorem is given after the following two lemmas.
\begin{lem}\label{T918}
Let $k\leq n\in \N$ and let $\s _i\in \Sent _{L_b}$ and $\Con(\t ^{\rm tree}\land \s _i)$, for $i<k$.
Then there exists a $L_b$-sentence $\p _{\s _0, \dots,\s _{k-1}}$
such that for each $L_b$-structure $\Y$ we have
\begin{equation}\label{EQ964}\textstyle
\Y \models \t_{\CC _n} \land \p _{\s _0, \dots,\s _{k-1}} \Leftrightarrow \exists \X _0,\dots,\X _{n-1}\in \CC ^{\rm rt}_{>1}\;(\Y \cong \prod _{i<n}\X _i \land \forall i<k \;\X _i^+\models \s _i) .
\end{equation}
\end{lem}
\dok
We prove that (\ref{EQ964}) is true for the sentence
\begin{equation}\label{EQ963}\textstyle
\p _{\s _0, \dots,\s _{k-1}}:=\exists w_0,\dots , w_k \;(\bigwedge _{i<k}\f_1 (w_i)\land \bigwedge _{i<j<k}\neg \r (w_i,w_j) \land \bigwedge _{i<k} \s _i ^{\r (w_i,v)}(w_i)) .
\end{equation}
Let $\Y \models \t_{\CC _n} \land \p _{\s _0, \dots,\s _{k-1}}$.
Then by Theorem \ref{T915} we have $\Y \in \CC _n$
and, hence, $\Y \cong \X:=\prod _{i<n}\X _i$, where $\X _i\in \CC ^{\rm rt}_{>1}$, for $i<n$.
Since $\Y\models \p _{\s _0, \dots,\s _{k-1}}$ we have $\X \models \p _{\s _0, \dots,\s _{k-1}}$,
which means that there are $a_0,\dots,a_k \in X$ such that

(i) $\X \models \f _1[a_j]$, for $j<k$,

(ii) $\X \models \neg \r [a_j,a_{j'}]$, for $j<j'<k$,

(iii) $\X \models \s _j ^{\r (w_j,v)}[a_j]$, for $j<k$.

\noindent
By (i), Fact \ref{T531}(a) and (\ref{EQ591}) we have  $a_0,\dots,a_k \in X^{(1)}=\bigcup _{i<n}A_i$;
so, for each $j<k$ we have $a_j\in A_{i_j}$, for some $i_j<n$,
and by Fact \ref{T531}(b) $A_{i_j}=\{ x\in X :\X\models \r[a_j,x]\}=D_{\r(a_j,v),\X}$.
By (iii) for each $j<k$ we have $\X \models \s _j ^{\r (w_j,v)}[a_j]$,
which by Fact \ref{T704} gives $D_{\r(a_j,v),\X}\models \s _j$.
Thus $A_{i_j}\models \s _j$,
and, since by Fact \ref{T531}(c) we have $A_{i_j}\cong \X _{i_j}^+$,
we obtain $\X _{i_j}^+\models \s _j$.
By (ii) and Fact \ref{T531}(b) $i_j \neq i_{j'}$, for different $j,j'<k$,
thus $p:k\rightarrow \{ i_j:j<k\}$ is a bijection.
Let $\pi\in \Sym (n)$, where $p\subset \pi$
and let $\Z _j:=\X _{\pi (j)}$, for $j<n$.
Then for $j<k$ we have  $\Z _j:=\X _{\pi (j)}=\X _{p(j)}=\X _{i_j}$
so $\Z _j^+=\X _{i_j}^+ \models \s _j$.
Now by Fact \ref{T043}(a) $\Y \cong \prod _{j<n}\Z _j$
and, since $\Z_0,\dots,\Z _{n-1}\in \CC ^{\rm rt}_{>1}$, the r.h.s. of (\ref{EQ964}) is true.

Conversely, let the r.h.s. of (\ref{EQ964}) hold for $\Y$ and let $\X :=\prod _{i<n}\X _i$.
Then by Theorem \ref{T915} we have $\Y \models \t_{\CC _n}$ and we prove that $\Y \models  \p _{\s _0, \dots,\s _{k-1}}$.
By Fact \ref{T531}(b) we have $X^{(1)}=\bigcup _{i<n}A_i$;
let us take $a_j\in A_j$, for $j<k$.
Then by Fact \ref{T531}(a) for $j<k$ we have $a_j\in X^{(1)}=D_{\f _1,\X}$ and (i) is true.
For $j<j'<k$ by Fact \ref{T531}(b) we have $\X \models \neg \r [a_j,a_{j'}]$ and (ii) is true.
Since $\X _i^+\models \s _j$ and
by Fact \ref{T531}(b) and (c) we have $D_{\r(a_j,v),\X}=A_{i_j}\cong \X _{i_j}^+$
we have $D_{\r(a_j,v),\X}\models \s _j$
and, by Fact \ref{T704}, $\X \models \s _j ^{\r (w_j,v)}[a_j]$; so, (iii) is true.
Thus, $\X \models \p _{\s _0, \dots,\s _{k-1}}$ and, hence, $\Y \models \p _{\s _0, \dots,\s _{k-1}}$ too.
\hfill $\Box$
\begin{lem}\label{T938}
For each $\t\in \Sent _{L_b}$, where $\Con (\t^{\rm rt}_{>1}\land \t)$, there is $\s _\t\in \Sent _{L_b}$ such that
\begin{equation}\label{EQ977}
\forall \X \in \CC ^{\rm rt}_{>1} \;(\X \models \t \;\mbox{ iff }\;\X ^+\models \s _\t ).
\end{equation}
\end{lem}
\dok
We regard the $0$-th Lindenbaum-Tarski algebra of $\{\t^{\rm rt}_{>1}\}$,  $\B:=\Sent _{L_b}/\!\leftrightarrow _{\t^{\rm rt}_{>1}}$,
its Stone dual $\S:=\{ \tilde\CT: \t^{\rm rt}_{>1}\in \CT \mbox{ and $\CT$ is a complete $L_b$-theory}\}$,
where $\tilde\CT:= \{ [\p] :\p \in \CT\}$,
and its second dual $\CB:=\{ B_{[\p]}: [\p]\in \B\}\cong \B$, a clopen base for the topology on $\S$,
where $B_{[\p]}:=\{ \tilde \CT \in \S: [\p]\in \tilde \CT \}$.

Clearly, if $\X\models \t^{\rm rt}_{>1}$, then $\X ^+=D_{\f (v),\X}$, where $\f (v):= \exists u\; (\forall w \;(u\leq w)\land v>u)$.

By the assumption we have $B_{[\t^{\rm rt}_{>1}\land \t]} \neq \emptyset$.
Let $\CT\in B_{[\t^{\rm rt}_{>1}\land \t]}$ and $\X \models \CT$.
We show that $\CT_\X :=\{\t^{\rm rt}_{>1}\}\cup \{ \p ^\f:\p\in \Th (\X ^+)\}$ is an axiomatization of $\CT =\Th (\X)$.
By Fact \ref{T704} for $\p\in \Th (\X ^+)$ we have $\X \models \p ^\f$;
thus $\CT _\X \subset \CT$ and, hence, $\Mod (\CT)\subset \Mod (\CT_\X)$.
Conversely, if $\Y\models \CT _\X$, then $\Y$ is a rooted tree of size $>1$
and for each $\p\in \Th (\X ^+)$  we have $\Y \models \p ^\f$ and, by Fact \ref{T704}, $\Y ^+\models \p$.
Thus $\Y ^+\equiv \X ^+$ and, by Fact \ref{T200}(b), $\Y \equiv \X $, that is $\Y \in \Mod (\CT)$.

So, $\Mod (\CT)= \Mod (\CT _\X)$
and, by Fact \ref{T922}, $\X,\Y\models \CT$ implies that $\Th (\X^+)=\Th (\Y^+)=:\CT ^+$,
thus $\CT_\X=\CT_\Y$ and $\CT' :=\{\t^{\rm rt}_{>1}\}\cup \{ \p ^\f:\p\in \CT ^+\}$ is an axiomatization of $\CT$.
Consequently, $\CB (\CT):=\{B_{[\t^{\rm rt}_{>1}]}\}\cup \{ B_{[\p ^\f]}:\p\in \CT ^+\}$
is a clopen neighborhood base of the point $\CT$ of $\S$.
Thus, since $B_{[\t^{\rm rt}_{>1}\land \t]}$ is a neighborhood of $\CT$,
there is $B_{[\eta _\CT]}\in \CB (\CT)$ such that $\CT \in B_{[\eta_\CT]}\subset B_{[\t^{\rm rt}_{>1}\land \t]}$,
which means that in $\B$ we have $[\eta _\CT]\leq [\t^{\rm rt}_{>1}\land \t]$.

If $\eta_\CT =\t^{\rm rt}_{>1}$, for some $\CT\in B_{[\t^{\rm rt}_{>1}\land \t]}$,
then $[\t^{\rm rt}_{>1}]\leq [\t^{\rm rt}_{>1}\land \t]\leq[\t^{\rm rt}_{>1}]$,
which gives $[\t^{\rm rt}_{>1}\land \t]=[\t^{\rm rt}_{>1}]$.
So, defining $\s _\t := \forall v\; v=v$,
for $\X \in \CC ^{\rm rt}_{>1}$ we have
$\X \models \t $
iff $\X \models \t^{\rm rt}_{>1}\land \t$
iff $\X \models \t^{\rm rt}_{>1}$
which is true;
since $\X ^+\models \s _\t $ is true too, we have (\ref{EQ977}) and the lemma is proved.

The remaining case is when for each $\CT\in B_{[\t^{\rm rt}_{>1}\land \t]}$
there is $\p _{\CT}\in \CT ^+$ such that  $\CT \in B_{[\p _{\CT}^\f]}\subset B_{[\t^{\rm rt}_{>1}\land \t]}$.
Then $B_{[\t^{\rm rt}_{>1}\land \t]}=\bigcup_{\CT\in B_{[\t^{\rm rt}_{>1}\land \t]}} B_{[\p _{\CT}^\f]}$ is an open cover of the compact set $B_{[\t^{\rm rt}_{>1}\land \t]}$
and, hence, there is a finite subcover $B_{[\t^{\rm rt}_{>1}\land \t]}=\bigcup_{j<k} B_{[\p _{\CT_k}^\f]}$,
which, since $\CB \cong \B$, gives $[\t^{\rm rt}_{>1}\land \t]= [\bigvee_{j<k} \p _{\CT_k}^\f] = [(\bigvee_{j<k} \p _{\CT_k})^\f]=[\s _\t ^\f]$,
where $\s_\t:=\bigvee_{j<k} \p _{\CT_k}$;
in other words, $\t^{\rm rt}_{>1}\vdash \t^{\rm rt}_{>1}\land \t \Leftrightarrow \s_\t ^\f$.
Now for $\X \in \CC ^{\rm rt}_{>1}$ we have
$\X \models \t $
iff $\X \models \t^{\rm rt}_{>1}\land \t$
iff $\X \models \s _\t^\f $
iff (by Fact \ref{T704}) $\X ^+ \models \s_\t $,
and (\ref{EQ977}) is true again.
\kdok
\noindent
{\bf Proof of Theorem \ref{T939}.}
Let $k\leq n \in \N$ and $\t _0, \dots,\t _{k-1}\in \Sent _{L_b}$.
If $\Con(\t ^{\rm rt}_{>1}\land \t _i)$ fails for some $i<k$,
then the class $\CC_n ^{\t _0, \dots,\t _{k-1}}=\emptyset$ is definable by a contradiction.

Otherwise for $i<k$ we have $\Con(\t ^{\rm rt}_{>1}\land \t _i)$
and by Lemma \ref{T938} there is $\s _{\t _i}\in \Sent _{L_b}$ such that
\begin{equation}\label{EQ978}
\forall \X \in \CC ^{\rm rt}_{>1} \;(\X \models \t_i \;\mbox{ iff }\;\X ^+\models \s _{\t_i}  ) .
\end{equation}
Since $\Con(\t ^{\rm rt}_{>1}\land \t _i)$ there is $\X \in \CC ^{\rm rt}_{>1}$ such that $\X \models \t_i$
and by (\ref{EQ978}) we have $\X ^+\models \s _{\t_i}$;
so, since $\X ^+\models \t ^{\rm tree}$, we have $\Con (\t ^{\rm tree}\land \s _{\t_i})$.
Applying Lemma \ref{T918} to the formulas $\s _{\t_i}$ (instead of $\s _i$), $i<k$, we obtain the formula (see (\ref{EQ963}))
\begin{equation}\label{EQ980}\textstyle
\p _{\s_{\t _0}, \dots,\s_{\t _{k-1}}}:=\exists w_0,\dots , w_k \;(\bigwedge _{i<k}\f_1 (w_i)\land \bigwedge _{i<j<k}\neg \r (w_i,w_j) \land \bigwedge _{i<k} \s_{\t _i} ^{\r (w_i,v)}(w_i)),
\end{equation}
and by (\ref{EQ964}) for each $\Y\in \Mod_{L_b}$ we have
$\Y \models \t_{\CC _n} \land \p _{\s_{\t _0}, \dots,\s_{\t _{k-1}}}$
iff there are $\X _0,\dots,\X _{n-1}\in \CC ^{\rm rt}_{>1}$ such that $\Y \cong \prod _{i<n}\X _i$
and $\X _i^+\models \s_{\t _i}$, that is, (by \ref{EQ978}) $\X _i\models \t _i$, for $i<k$;
iff $\Y \in \CC_n ^{\t _0, \dots,\t _{k-1}}$.
Thus the sentence $\t_{\CC _n} \land \p _{\s_{\t _0}, \dots,\s_{\t _{k-1}}}$ defines the class $\CC_n ^{\t _0, \dots,\t _{k-1}}$.
\hfill $\Box$
\begin{cor}\label{T919}
For each $n\in \N$ and $\t \in \Sent _{L_b}$ the sentence $\t_{\CC _n}\land \p _{\s_\t , \dots,\s_\t}$ defines the class
$$
\CC_n^\t := \{ \Y \in\Mod_{L_b}:\Y \mbox{ is isomorphic to a direct product of $n$ rooted trees of size $>1$ satisfying $\t$}\}.
$$
In particular, the following class of distributive lattices is first-order definable
$$
\CC_n^{\rm lo} :=  \{ \Y \in\Mod_{L_b}:\Y \mbox{ is isomorphic to a direct product of $n$ linear orders with $0$ of size $>1$}\}
$$
and the same holds if we take linear orders with additional first-order properties; e.g.\ bounded, dense, discrete, etc.
\end{cor}
\begin{ex}\label{EX101}\rm
For $m\in \N$, a tree $\X$ is said to have {\it width $\leq m$} (resp. {\it height $\leq m$}) if it has no antichain (resp. chain) of size $m+1$.
$\X$ is {\it m-branching} iff for each $x\in X$ there are pairwise incomparable $y_0,\dots ,y_{m-1}>x$;
$\X$ is {\it leafless} iff it is 1-branching (has no maximal elements).
Clearly each of these properties is expressible by a first order sentence $\t$
and by Corollary \ref{T919}
for each $n\in \N $ the corresponding class $\CC_n^\t$ of partial orders is first-order definable.
\end{ex}
\section{FLD theories with finitely axiomatizable decompositions}\label{S6}
If $\CT$ is an FLD$_0$-theory of partial order, a decomposition $\sum _{\BI}\X_i \in \CD (\CT)$
will be called a {\it finitely axiomatizable decomposition} iff
the summands $\X _i$, $i<n$, are finitely axiomatizable.
Here we show that if such a decomposition exists, then $\CT$ is finitely axiomatizable, inherits VC and VC$^{\sharp}$ from the summands
and we describe all models of $\CT$.

\begin{te}\label{T932}
If $\CT$ is an FLD$_0$-theory with a finitely axiomatizable decomposition $\sum _{\BI}\X_i \in \CD (\CT)$, then

(a) The theory $\CT$ is finitely axiomatizable;

(b) $\Y\models \CT$ iff $\Y =\sum _{\BI}\Y_i$, where $\Y_i \equiv \X _i$, for all $i<n$;

(c) If the summands $\X_i$, $i<n$, satisfy VC, then $\CT$ satisfies VC and (\ref{EQ754}) holds;

(d) If the summands $\X_i$, $i<n$, satisfy VC$^{\sharp}$, then $\CT$ satisfies VC$^{\sharp}$: $I(\CT)=1$, if $I(\X_i)=1$, for all $i<n$, otherwise we have $I(\CT)=\c$.
\end{te}
\dok
(a) By the assumption, $\BI=\la n,\leq _{\BI}\ra$ is a finite partial order and $\X _i$, $i<n$, are partial orders with a smallest element.
We assume that $X_i$, $i<n$, are pairwise disjoint sets.
Let $r_i:=\min \X _i$, for $i<n$, and $\bar r:=\la r_0,\dots ,r_{n-1}\ra$.
Since $\X =\sum _{\BI}\X_i$ is an FLD$_0$-partial order,
by Fact \ref{T700} there are $L_b$-formulas $\f_i (w_0,\dots ,w_{n-1},v)=\f_i (\bar w,v)$, $i<n$,
such that (a) and (b) are true.
For $i<n$, let $\t _i \in \Th (\X _i)$ be a sentence such that $\Y\models \t _i$ implies $\Y\equiv \X_i$, for each $\Y\in \Mod _{L_b}$.
Let $\p _{\BI} (w_0,\dots,w_{n-1})=\p _{\BI}(\bar w)$ be the $L_b$-formula
$$\textstyle
\p_{\BI}(w_0,\dots,w_{n-1}):= \bigwedge _{\f (\bar w)\in\; \Lit _{L_b} \land \;\BI\;\models \;\f [0,\dots,n-1]} \f (\bar w).
$$
\begin{cla}\label{T933}\rm
If $\Y$ is an $L_b$-structure  and $y_0, \dots ,y_{n-1}\in Y$,
then
\begin{center}
$f:=\{ \la i,y_i\ra :i<n\} : \BI \rightarrow \{ y_0, \dots ,y_{n-1}\}\subseteq \Y$ is an isomorphism iff
$\Y \models \p _{\BI} [\bar y]$.
\end{center}
\end{cla}
\dok
If $L_n=L_b\cup \la c_0,\dots,c_{n-1}\ra$, where $c_i$, $i<n$, are new different constants,
$\BI_n :=(\BI, \la 0,\dots,n-1 \ra)$ is the corresponding expansion of $\BI$
and $\Delta _{\BI}:=\{ \f (c_0,\dots,c_{n-1})\in \Lit _{L_n}:\BI_n\models \f (c_0,\dots,c_{n-1}) \}$
is the diagram of $\BI$
(the set of all atomic sentences and negations of atomic sentences of $L_n$ which are true in $\BI_n$),
then (see \cite{CK}, p.\ 69)
$f$ is an isomorphism iff $(\Y ,\la y_0, \dots ,y_{n-1} \ra)\models \Delta _{\BI}$
iff  $(\Y ,\bar y)\models \bigwedge \{ \f (\bar c):\f (\bar c)\in \Lit _{L_n}\land \;\BI_n\models \f (\bar c) \}$
iff  $(\Y ,\bar y)\models \bigwedge \{ \f (\bar c):\f (\bar w)\in \Lit _{L_b}\land \;\BI\models \f [0,\dots,n-1] \}$
iff  $\Y \models \bigwedge \{ \f (\bar w):\f (\bar w)\in \Lit _{L_b}\land \;\BI\models \f [0,\dots,n-1] \}[\bar y]$.
\kdok
\noindent
A sentence $\t _\X$ axiomatizing $\Th (\X)$ will be a conjunction $\p ^{\rm po}\land \p$, where
for an $L_b$-structure $\Y$ the sentence $\p ^{\rm po}$ says that $\Y$ is a partial order
and $\p$ says that
there are $y_0, \dots ,y_{n-1}\in Y$ such that

({\sc i}) $f:=\{ \la i,y_i\ra :i<n\} : \BI \rightarrow \{ y_0, \dots ,y_{n-1}\}\subseteq \Y$ is an isomorphism,

({\sc ii}) $\{D_{\f _i (\bar y,v),\Y} :i<n \}$ is a partition of the domain $Y$ of $\Y$,

({\sc iii}) $y_i\in D_{\f _i (\bar y,v),\Y}$ and $y_i=\min D_{\f _i (\bar y,v),\Y}$, for $i<n$,

({\sc iv}) $D_{\f _i (\bar y,v),\Y}\models \t _i$, that is, $D_{\f _i (\bar y,v),\Y}\equiv \X_i$, for $i<n$,

({\sc v}) $\Y$ is the $\{ y_i :i<n\}$-lexicographic sum $\sum _{\{ y_i :i<n\}}D_{\f _i (\bar y,v),\Y}$;

\noindent
for example, let
\begin{eqnarray}
\p & := &\textstyle \exists w_0, \dots ,w_{n-1} \Big[ \p _{\BI}(\bar w) \land \forall u \bigvee _{i<n} \f _i (\bar w,u)\land
                                         \bigwedge _{\{i,j\}\in [I]^2} \neg\exists u \;( \f _i (\bar w,u)\land \f _j (\bar w,u))\land \nonumber\\
   &    &\textstyle \bigwedge _{i<n}  \Big(\f _i (\bar w,w_i) \land \forall w \;(\f _i (\bar w,w)\Rightarrow w_i \leq w)\land \t_i ^{\f _i(\bar w,v)}\Big) \land \label{EQ972}\\
   &    &\textstyle \bigwedge _{\{i,j\}\in [I]^2}\forall u,v \;\Big(\f _i (\bar w,u)\land \f _j (\bar w,v)\Rightarrow (u\leq v \Leftrightarrow w_i \leq w_j)\Big)\Big]\nonumber.
\end{eqnarray}
\noindent
Now, $\X\models \p ^{\rm po}\land \p$, because taking $y_i:=r_i$, for $i<n$, by Fact \ref{T700}(a)  ({\sc i})--({\sc v}) are true; so, $\t _\X \in \Th (\X)$.

Conversely, assuming that $\Y \models \t _\X$, we prove that $\Y \equiv \X$.
So $\Y$ is a partial order, $\Y \models \p$
and taking $y_0, \dots ,y_{n-1}\in Y$ provided by (\ref{EQ972}) we have ({\sc i})--({\sc v}).
By ({\sc v}) we have $\Y=\sum _{\{ y_i :i<n\}}D_{\f _i (\bar y,v),\Y}$
and, since by ({\sc i}) the mapping $i\mapsto y_i$ is an isomorphism from $\BI =\la n,\leq _\BI\ra$ onto  $\{ y_i:i<n\}\subseteq \X$,
we have $\Y=\sum _{\BI}D_{\f _i (\bar y,v),\Y}$.
By ({\sc iv}) we have $D_{\f _i (\bar y,v),\Y}\models \t_i$;
more precisely, by (\ref{EQ972}) we have $\Y \models \t_i ^{\f _i(\bar w,v)}[\bar y]$,
which by Fact \ref{T704} gives $D_{\f _i (\bar y,v),\Y}\models \t_i$.
Thus $D_{\f _i (\bar y,v),\Y}\equiv \X_i$, for $i<n$,
and, by Fact \ref{T200}(b), $\Y= \sum _{\BI}D_{\f _i (\bar y,v),\Y} \equiv \sum _{\BI}\X _i=\X$.

(b) If $\Y\equiv \sum _{\BI}\X_i$, then, taking finite axiomatizations $\t _i$ of $\X _i$, for $i<n$,
by Fact \ref{T700}(b) we obtain a decomposition $\Y=\sum _{\BI}\Y _i$
such that for each $i<n$ we have $\Y_i \models \t _i$ and, hence, $\Y_i \equiv \X _i$.
The converse follows from Fact \ref{T200}(b).

(c) If $\X_i$-s satisfy VC, then the theory $\CT$ is actually Vaught's (see (\ref{EQ741}).
Namely, for $i<n$ we have $I(\X _i)\leq \o$ or $I(\X _i)= \c$;
so, if $\Z\models \t_i$, then $\Th (\Z)=\Th (\X _i)$
and, hence, $I(\Z)\leq \o$ or $I(\X)= \c$.
Now, the claim follows from Fact \ref{T712}.

(d) If $I(\X _i)\in \{ 1,\c\}$, for all $i<n$, then $\prod _{i<n}I(\X _i)\in \{ 1,\c\}$ and the claim follows from (c).
\hfill $\Box$
\begin{ex}\label{EX100}
A finitely axiomatizable FLD$_0$-theory having a decomposition with a summand which is not finitely axiomatizable. \rm
By Fact \ref{T950}(b) the linear order $\X =\X_0 +\X _1 +\X _2$, where $\X _0\cong 1$ and $\X_1,\X _2 \cong \Z$ is finitely axiomatizable.
But there is a decomposition $\X =\Y_0 +\Y _1 +\Y _2$, where $\Y _0\cong 1 +\o ^*$, $\Y _1\cong \o +\o ^*$ and $\Y _2\cong \o$,
so $\Y _i$, $i<3$, are partial orders having a smallest element and, by Fact \ref{T950}(c), $\Y_1$ is not finitely axiomatizable.
\end{ex}
Clearly, for FLD$_1$-theories we have a dual of Theorem \ref{T932}.
\section{Sums of products of finitely axiomatizable rooted trees $\la\la\CC ^{\rm rt}_{\rm fa} \ra_\Pi\ra_\Sigma$}\label{S7}
Let $\CC ^{\rm rt}_{\rm fa}$ denote the class of finitely axiomatizable rooted trees
and let  $\la\la\CC ^{\rm rt}_{\rm fa} \ra_\Pi\ra_\Sigma$ be its closure described in Fact \ref{T944}.
Since $\CC ^{\rm rt}_{\rm fa}\subset\CC ^{\rm rt}$, by Theorem \ref{T907} Vaught's conjecture is true for each poset $\X\in \la\la\CC ^{\rm rt}_{\rm fa} \ra_\Pi\ra_\Sigma$.
Here we show that, in addition, such $\X$ is finitely axiomatizable and describe all models of $\Th (\X )$.
\begin{te}\label{T943}
If $\X\in \la\la\CC ^{\rm rt}_{\rm fa} \ra_\Pi\ra_\Sigma$
and $\X \cong \sum _\BI \prod _{j<m_i}\X_{i,j}$, where  $\BI \in \CC ^{\rm fin}$ and  $\X_{i,j}\in \CC ^{\rm rt}_{\rm fa}$, for $j<m_i$ and $i\in I$, then

(a) $\X$ is finitely axiomatizable;

(b) $\Y \equiv \X$ iff $\Y \cong \sum _{\BI }\prod _{j<m_i}\Y_{i,j}$, where $\Y_{i,j}\equiv \X_{i,j}$, for each $j<m_i$ and $i\in I$.

\noindent
In addition, $\CT :=\Th (\X)$ satisfies Vaught's conjecture and (a), (b) and (c) of Theorem \ref{T907} are true.
\end{te}
\dok
First we consider the closure of $\CC ^{\rm rt}_{\rm fa}$ under products, $\la\CC ^{\rm rt}_{\rm fa} \ra_\Pi$ and prove the following claim.
\begin{cla}\label{T921}
If $\X =\prod _{i<n}\X_i$, where $\X _i\in \CC ^{\rm rt}_{\rm fa}$, for $i<n$, then $\X$ is finitely axiomatizable.
\end{cla}
\dok
We first regard the case when $\X _i\in \CC ^{\rm rt}_{>1}$, for all $i<n$,
and assuming that $\t _i\in \Th (\X_i)$ is a finite axiomatization of $\Th (\X_i)$, for $i<n$,
prove that $\t_{\CC _n} \land \p _{\s_{\t_0},\dots, \s_{\t_{n-1}}}$  (see (\ref{EQ980})) is a finite axiomatization of $\Th (\prod _{i<n}\X_i)$.
So, we have to prove that for each $\Y \in \Mod _{L_b}$ we have $\Y\models \t_{\CC _n} \land \p _{\s_{\t_0},\dots, \s_{\t_{n-1}}}$ iff $\Y \equiv \prod _{i<n}\X_i$;
that is, by Theorem \ref{T939} and Fact \ref{T544}(a), that (i) $\Leftrightarrow$ (ii), where

(i) $\exists \Y _0,\dots,\Y _{n-1}\in \CC ^{\rm rt}_{>1}\;\;(\Y \cong \prod _{i<n}\Y _i \land \forall i<n \;\Y _i\models \t _i)$,

(ii) $\exists \Y _0,\dots,\Y _{n-1}\in \Mod _{L_b}\;\;(\Y \cong \prod _{i<n}\Y _i \land \forall i<n \;\Y _i\equiv \X _i)$.

\noindent
Assuming (i) for each $i<n$ we have $\Y _i\models \t _i$ and, hence, $\Y _i\equiv \X _i$; so, (ii) is true.
Assuming (ii) for each $i<n$ we have $\X _i\in \CC ^{\rm rt}_{>1}$ and $\Y _i\equiv \X_i$,
which gives $\Y _i\in \CC ^{\rm rt}_{>1}$ and $\Y _i\models \t _i$;
so, (i) is true.

Now we consider the general situation.
Let $J:=\{ i<n : |X_i|>1\}$.
If $J=\emptyset$, then $|X|=1$ and, clearly, $\X$ is finitely axiomatizable.
Otherwise we have $\X \cong \prod _{i\in J}\X _i$
and $\X_i \in \CC ^{\rm rt}_{>1}$, for $i\in J$.
As above, $\prod _{i\in J}\X _i$ is finitely axiomatizable
and, clearly, $\Th (\X)=\Th (\prod _{i\in J}\X _i)$.
\kdok
Now we prove the theorem.
Since $\X \cong \X ':=\sum _\BI \prod _{j<m_i}\X_{i,j}$, which implies $\Th (\X)=\Th (\X')$,
w.l.o.g.\ we assume that $\X = \sum _\BI \prod _{j<m_i}\X_{i,j}$.

(a) By Claim \ref{T921} the products $\X _i:=\prod _{j<m_i}\X_{i,j}$, $i<n$, are finitely axiomatizable.
Since $\min \X _i$ exists for each $i\in I$,
$\X =\sum _\BI \X_i$ is a FLD$_0$ poset
and by Theorem \ref{T932}(a) $\X$ is finitely axiomatizable.

(b) If $\Y \equiv \X$,
then by Theorem \ref{T932}(b) $\Y = \sum _{\BI }\Z_i$, where $\Z_i \equiv \X _i$, for $i<n$.
Thus for each $i<n$ we have $\Z_i \equiv \prod _{j<m_i}\X_{i,j}$
and, by Fact \ref{T544}(a), $\Z_i \cong \prod _{j<m_i}\Y_{i,j}$, where $\Y_{i,j}\equiv \X_{i,j}$, for each $j<m_i$.
So, by Fact \ref{T200}(a) we have $\Y = \sum _{\BI }\Z_i\cong \sum _{\BI }\prod _{j<m_i}\Y_{i,j}$.
Conversely, let $\Y \cong \sum _{\BI }\Y _i$,
where for each $i\in I$ we have $\Y _i=\prod _{j<m_i}\Y_{i,j}$ and $\Y_{i,j}\equiv \X_{i,j}$, for each $j<m_i$.
Then by Fact \ref{T043}(c) for $i\in I$ we have  $\Y _i\equiv\prod _{j<m_i}\X_{i,j}=:\X_i$
and, by Fact \ref{T200}(b),   $\Y \cong \sum _{\BI }\Y _i \equiv \sum _{\BI }\X _i=\X$.
\kdok
Concerning the applications of Theorem \ref{T943}, the following statement shows that $\X_r\in\CC ^{\rm rt}_{\rm fa}$, for each finitely axiomatizable tree $\X$.
\begin{prop}\label{T920}
If $\X \in \CC ^{\rm rt}_{>1}$, then $\Th (\X)$ is finitely axiomatizable iff $\Th (\X ^+)$ is finitely axiomatizable.
\end{prop}
\dok
Assuming that $\Y\equiv \X$ iff $\Y\models \t$, for all $\Y\in \Mod _{L_b}$, we show that $\Z\equiv \X ^+$ iff $\Z\models \t ^{\rm tree} \land \s _\t$.
If $\Z\equiv \X ^+$, then $\Z\models \t ^{\rm tree}$
and, by Fact \ref{T200}(b), $\Z _r \equiv \X$,
which gives $\Z _r \models \t$ and $\Z_r \in \CC ^{\rm rt}_{>1}$;
so, by Lemma \ref{T938}, $\Z =(\Z_r)^+\models  \s _\t$.
Conversely, if $\Z\models \t ^{\rm tree} \land \s _\t$,
then $\Z_r \in \CC ^{\rm rt}_{>1}$ and $(\Z _r)^+ \models  \s _\t$,
by Lemma \ref{T938} we have $\Z _r \models  \t$,
by the assumption $\Z_r \equiv \X$
and, by Fact \ref{T922}, $\Z\equiv \X ^+$.

Assuming that $\Z\equiv \X ^+$ iff $\Z\models \s $, we show that $\Y\equiv \X$ iff $\Y\models \t^{\rm rt}_{>1}\land \s ^\f$.
If $\Y\equiv \X$, then $\Y\models \t^{\rm rt}_{>1}$
and, by Fact \ref{T922}, $\Y^+\equiv \X ^+$,
which gives $D_{\f,\Y}=\Y ^+\models \s$;
so, by Fact \ref{T704}, $\Y \models \s^\f$.
Conversely, if $\Y\models \t^{\rm rt}_{>1}\land \s ^\f$,
then $\Y^+\models \s $;
so $\Y ^+ \equiv \X ^+$
and, by Fact \ref{T200}(b), $\Y\equiv \X$.
\hfill $\Box$
\paragraph{Finitely axiomatizable $\o$-categorical theories}
The following application of Theorem \ref{T943} is related to a general question: Which $\o$-categorical theories are finitely axiomatizable?
We recall some classical results.
\begin{fac}\label{T946}
If $\X$ is an $\o$-categorical partial order, then we have

(a) (Rosenstein \cite{Rosen1}) If $\X$ is a linear order, then $\X$ is finitely axiomatizable;

(b) (Schmerl  \cite{Sch0}) If $\X$ is a tree, then $\X$ is finitely axiomatizable iff $\X$ is finite-branching;

(c) (Schmerl \cite{Sch1}) If $\X$ is a partial order of finite width, then $\X$ is finitely axiomatizable.
\end{fac}
We recall that a partial order $\X$ is of {\it finite width} iff there is $n>1$ such that $\X$ has no antichains of size $n$,
and that finite-branching trees were introduced by Schmerl in \cite{Sch0}.
Since in this paper we work with rooted trees we note that a tree $\X\in \CC ^{\rm rt}$ is finite-branching
iff there is $n\in \N$ such that whenever $x,y\in X$

- the subtree $(x,\cdot)$ has $\leq n$ connectivity components,\footnote{The connectivity components of a tree $\X$ are the equivalence classes
corresponding to the equivalence relation $\sim$ on the set $X$ defined by: $x\sim y$ iff there exists $z\leq x,y$.
The binary tree $2^{<\o}$ is finite-branching, but it is not of finite width.}

- the subtree $\{ z\in X : z>(\cdot,x)\cap (\cdot,y)\}$ has $\leq n$ components, if $x$ and $y$ are incomparable.

\noindent
By \cite{Sch0} the class of finite-branching trees is first-order axiomatizable and, hence, $\cong$-closed.
So, the same holds for the class $\CC^{\rm rt}_{\rm fb}$ of finite-branching rooted trees; let $\la \la\CC^{\rm rt}_{\rm fb}\ra_\Pi\ra_\Sigma$ be its closure from Fact \ref{T944}(b).
\begin{te}\label{T947}
For each  partial order $\X \in \la \la\CC^{\rm rt}_{\rm fb}\ra_\Pi\ra_\Sigma$ we have
$$
\X \mbox{ is $\o$-categorical} \Rightarrow \X  \mbox{ is finitely axiomatizable}.
$$
\end{te}
\dok
Since $\X \in \la \la\CC^{\rm rt}_{\rm fb}\ra_\Pi\ra_\Sigma$,
by Fact \ref{T944}(b) we have $\X \cong \sum _\BI \prod _{j<m_i}\X_{i,j}$,
where  $\BI \in \CC ^{\rm fin}$ and $\X_{i,j}\in \CC ^{\rm rt}_{\rm fb}$, for $j<m_i$ and $i\in I$.
Since $\X$ is $\o$-categorical,
$\sum _\BI \prod _{j<m_i}\X_{i,j}$ is $\o$-categorical too,
by Theorem \ref{T907}(b) for each $i\in I$ and $j<m_i$  the poset $\X_{i,j}$ is an $\o$-categorical tree,
and since $\X_{i,j}$ is finite-branching, by Fact \ref{T946}(b) $\X_{i,j}$ is finitely axiomatizable; that is $\X_{i,j}\in \CC ^{\rm rt}_{\rm fa}$.
By Theorem \ref{T943}(a) $\X$ is finitely axiomatizable.
\kdok
Concerning Theorems \ref{T943} and \ref{T947} we note that by Fact \ref{T946}(b) $\CC^{\rm rt}_{\rm fa}\cap\CC^{\rm rt}_{\o{\rm-cat}}=\CC^{\rm rt}_{\rm fb}\cap\CC^{\rm rt}_{\o{\rm-cat}}$,
where $\CC^{\rm rt}_{\o{\rm-cat}}$ is the class of $\o$-categorical rooted trees,
and that $\CC^{\rm rt}_{\rm fa}\subset \CC^{\rm rt}_{\rm fb}$, which follows from the next general consequence of Schmerl's results from \cite{Sch0}.
\begin{fac}\label{T949}
Each finitely axiomatizable tree $\X$ is finite-branching.
\end{fac}
\dok
We use the following facts:
(a) If a sentence $\s$ is true in some tree then it is true in some finite-branching tree (\cite{Sch0}, Corollary 2.6);
(b) For each $n$ there is a $\Pi _2$-sentence $\t _n$ defining the class of $\leq n$-branching trees (\cite{Sch0}, p.\ 124).
So, if $\X$ is a tree and $\s$ is a sentence axiomatizing $\Th (\X)$,
then by (a) there is a finite-branching tree $\Y$ such that $\Y\models \s$;
so $\Y$ is $\leq n$-branching, for some $n$,
and by (b) we have $\Y \models \t _n$.
Since $\Y \equiv \X$ we have $\X \models \t _n$,
and by (b) $\X$ is finite-branching.
\kdok
Let $\CC^{\rm rt}_{\rm fw}$ denote the class of rooted trees of finite width and $\la \la\CC^{\rm rt}_{\rm fw}\ra_\Pi\ra_\Sigma$ its closure from Fact \ref{T944}(b).
The proof of the following statement is a copy of the proof of Theorem \ref{T947}; we use Fact \ref{T946}(c).
\begin{te}\label{T948}
Each $\o$-categorical partial order $\X \in \la \la\CC^{\rm rt}_{\rm fw}\ra_\Pi\ra_\Sigma$ is finitely axiomatizable.
\end{te}
\section{Examples obtained from finitely axiomatizable chains}\label{S8}
In Theorem \ref{T943} the properties of the theory $\Th (\X)$, where $\X=\sum _\BI \prod _{j<m_i}\X_{i,j}\in \la\la\CC ^{\rm rt}_{\rm fa} \ra_\Pi\ra_\Sigma$
are described via the properties of the theories of the rooted trees $\X_{i,j}$.
Here we isolate a subclass $\CC \subset \CC ^{\rm rt}_{\rm fa}$ such that for $\X \in \la\la\CC \ra_\Pi\ra_\Sigma$
the same properties of $\Th (\X)$ are described via the properties of the theories of the finitely axiomatizable linear orders composing the trees $\X_{i,j}$.
Let $\CC ^{\rm lo}$ denote the class of linear orders.
\begin{fac}\label{T950}\rm
We recall basic facts concerning the class $\CC ^{\rm lo}_{\rm fa}$ of finitely axiomatizable linear orders.

(a) Let $\M $ be the smallest class of linear order types such that
(c1) $1\in \M $,
(c2) $\t +\t'\in \M $, if $\t ,\t'\in \M $,
(c3) $\t \o ,\t\o^* \in \M $, if $\t \in \M $,
(c4) $\s (F)\in \M$, if $F\in [\M]^{<\o}$.
Then for each $\X \in \CC ^{\rm lo}$ and each $n$ there is $\Y\in \M \cap \CC ^{\rm lo}_{\rm fa}$
such that $\X\equiv _n \Y$ (Amit and Shelah \cite{Amit}, Schmerl (unpublished) Myers \cite{My}, see also \cite{Rosen}, p.\ 256).
Thus, if $\X \in \CC ^{\rm lo}_{\rm fa}$, then $\X\equiv\Y$, for some $\Y\in \M$.

(b) $\M _2\subset \CC ^{\rm lo}_{\rm fa}$,
where $\M _2$ is the closure under (c1), (c2), (c4) and (c3') $\t \Z \in \M$, if $\t \in \M$;
(Rubin \cite{Rub}, see also \cite{Rosen}, p.\ 312).
The subclass  $\M _1$ of $\M _2$ obtained from (c1), (c2) and (c4)
is exactly the class of countable $\o$-categorical linear orders
(Rosenstein \cite{Rosen1}, see also \cite{Rosen}, p.\ 299).
In addition, $\a,\a^* \in \M \cap \CC ^{\rm lo}_{\rm fa}$, for each ordinal $\a <\o ^\o$,
while $\o ^\o \not\in \CC ^{\rm lo}_{\rm fa}$ (see \cite{Rosen}, p.\ 262, 257).

(c) $\o +\o ^* \not\in \CC ^{\rm lo}_{\rm fa}$ (see \cite{Rosen}, p.\ 253).
Moreover, if there are $a<b\in X$ such that $[a,b]\cong \o +\Z\Y +\o^*$,
where $\Y\in\CC ^{\rm lo}$ and $a$ has no immediate predecessor and $b$ has no immediate successor,
then $\X \not\in \CC ^{\rm lo}_{\rm fa}$ (see \cite{Rosen}, p.\ 337).
Thus the class $\CC ^{\rm lo}_{\rm fa}$ is not closed under $+$.
\end{fac}
Let $\CC ^{\rm rlo}_{\rm fa}$ denote the class of finitely axiomatizable linear orders with a smallest element.
Recall that  by $\X _r$ we denote a tree obtained from a tree $\X$ by adding a new element $r$ such that $r<x$, for all $x\in X$.
In particular, let $\emptyset _r\cong 1$. The following claim gives an alternative definition of the class $\CC ^{\rm rlo}_{\rm fa}$.
\begin{fac}\label{T941}
$\CC ^{\rm rlo}_{\rm fa} =\{ \BL _r : \BL \in \CC ^{\rm lo}_{\rm fa}\;\mbox{ or }\;\BL=\emptyset\}$.
\end{fac}
\dok
Let $\BL \in\CC ^{\rm rlo}_{\rm fa}$ and $r=\min \BL$. If $\BL \cong 1$, then $\BL =\emptyset _r$.
Otherwise $\BL ^+$ is a linear order, $\BL =(\BL ^+)_r$ is finitely axiomatizable
and, by Proposition \ref{T920}, $\BL ^+$ is finitely axiomatizable.
Conversely, $1\in \CC ^{\rm rlo}_{\rm fa}$ is evident
and if $\BL$ is a finitely axiomatizable linear order, then $\BL= (\BL _r)^+$,
and by Proposition \ref{T920}, $\BL _r$ is finitely axiomatizable;
so, $\BL _r\in \CC ^{\rm rlo}_{\rm fa}$.
\kdok
Clearly the class $\CC ^{\rm rlo}_{\rm fa}$ is $\cong$-closed and $ \CC ^{\rm rlo}_{\rm fa}\subset \CC ^{\rm rt}_{\rm fa}\cap \CC ^{\rm rt}_{\rm fmd}$,
where $\CC ^{\rm rt}_{\rm fmd}$ is the class of rooted FMD trees.
Let $\CC$ be the class of lexicographic $\BK$-sums of linear orders from $\CC ^{\rm rlo}_{\rm fa}$, where $\BK$ is a finite rooted tree
$$\textstyle
\CC :=\{ \sum _\BK \BL _k: \BK \in \CC ^{\rm rt}\cap \CC ^{\rm fin} \land \forall k\in K \;\BL _k\in\CC ^{\rm rlo}_{\rm fa}\}.
$$
The class $\CC$ is $\cong$-closed;
namely, if $f:\sum _\BK \BL _k \rightarrow \Y$ is an isomorphism,
then as in Fact \ref{T945} we have $\Y =\sum _\BK \Y _k$,
where $\Y _k\cong \BL _k $ and, hence, $\Y _k \in \CC ^{\rm rlo}_{\rm fa}$, for each $k\in K$.
By Fact \ref{T944}(b) we have
$$\textstyle
\X \in \la\la\CC\ra_\Pi\ra_\Sigma \;\;\mbox{ iff }\;\;\X \cong \sum _{\BI }\prod _{j<m_i}\sum _{\BK_{i,j}}\BL _{i,j,\,k}
$$
where $\BI \in \CC^{\rm fin}$, $m_i\in \N$, $\BK_{i,j}\in \CC ^{\rm fin}\cap \CC^{\rm rt}$ and $\BL _{i,j,\,k}\in \CC ^{\rm rlo}_{\rm fa}$ are pairwise disjoint.
\begin{te}\label{T940}
For each partial order $\X \in \la\la\CC\ra_\Pi\ra_\Sigma$, where $\X \cong \sum _{\BI }\prod _{j<m_i}\sum _{\BK_{i,j}}\BL _{i,j,\,k}$, we have

(a) $\X$ is finitely axiomatizable;

(b) $\Y \equiv \X$ iff $\Y \cong \sum _{\BI }\prod _{j<m_i}\sum _{\BK_{i,j}}\BL _{i,j,\,k}'$, where $\BL _{i,j,\,k}'\equiv \BL _{i,j,\,k}$, for all indices $i,j,k$;

(c) $\X$ is $\o$-categorical iff $\BL _{i,j,\,k}$ is $\o$-categorical, for all indices $i,j,k$; otherwise $I(\X)=\c$.
\end{te}
\dok
In order to use Theorem \ref{T943} we show that $\CC \subset \CC ^{\rm rt}_{\rm fa}$.
Let $\T:=\sum _\BK \BL _k \in \CC$.
Since $\min \BL _k$ exists, for each $k\in K$,
and $\BK$ is a rooted tree, $\T$ is a rooted tree.
Also, $\T $ is an FLD$_0$ partial order
and since the summands $\BL _k$ are finitely axiomatizable,
by Theorem \ref{T932}(a) $\T$ is finitely axiomatizable.
Thus $\T\in \CC ^{\rm rt}_{\rm fa}$ and $\CC \subset \CC ^{\rm rt}_{\rm fa}$ indeed.
Consequently, $\la\la\CC\ra_\Pi\ra_\Sigma\subset\la\la\CC ^{\rm rt}_{\rm fa}\ra_\Pi\ra_\Sigma$
and, hence $\X \in \la\la\CC ^{\rm rt}_{\rm fa}\ra_\Pi\ra_\Sigma$.

(a) Since $\X \in \la\la\CC ^{\rm rt}_{\rm fa}\ra_\Pi\ra_\Sigma$ by Theorem \ref{T943}(a) $\X$ is finitely axiomatizable.

(b) For $i\in I$ and $j<m_i$ the tree $\T _{i,j}:=\sum _{\BK_{i,j}}\BL _{i,j,k}\in \CC$ is an FLD$_0$ partial order.
Since the summands $\BL _{i,j,k}$, are finitely axiomatizable,
by Theorem \ref{T932}(b) we have
\begin{equation}\label{EQ984}\textstyle
\T \equiv \T _{i,j} \;\;\mbox{ iff }\;\; \T=\sum _{\BK_{i,j}}\BL _{i,j,\,k}', \mbox{ where }\BL _{i,j,\,k}'\equiv \BL _{i,j,\,k}, \mbox{ for each }k\in K_{i,j}.
\end{equation}
Since $\X \cong \sum _{\BI }\prod _{j<m_i}\T _{i,j}$
we have $\Y \equiv \X$
iff (by Theorem \ref{T943}(b)) $\Y \cong \sum _{\BI }\prod _{j<m_i}\Y_{i,j}$, where $\Y_{i,j}\equiv \T_{i,j}$, for each $j<m_i$ and $i\in I$,
which by (\ref{EQ984}) means that $\Y_{i,j}=\sum _{\BK_{i,j}}\BL _{i,j,\,k}'$, where $\BL _{i,j,\,k}'\equiv \BL _{i,j,\,k}$, for each $k\in K_{i,j}$.
Thus $\Y \cong \sum _{\BI }\prod _{j<m_i}\sum _{\BK_{i,j}}\BL _{i,j,\,k}'$ and (b) is true.

(c) For $i\in I$ and $j<m_i$ the summands $\BL _{i,j,k}$ of the sum $\T _{i,j}=\sum _{\BK_{i,j}}\BL _{i,j,k}$
are linear orders and, hence, satisfy VC$^\sharp$.
So, by Theorem \ref{T932}(d), $I(\T _{i,j})=1$ iff $I(\BL _{i,j,k})=1$, for all $k\in K_{i,j}$; otherwise we have $I(\T _{i,j})=\c$.
Since $\X \cong \sum _{\BI }\prod _{j<m_i}\T _{i,j}$,
by Theorem \ref{T907}(b) we have $I(\X )=1$ iff $I(\T _{i,j})=1$, for all $i\in I$ and $j<m_i$,
iff $I(\BL _{i,j,k})=1$, for all indices.
Otherwise, by Theorem \ref{T907}(c), $I(\X )=\c$.
\hfill $\Box$
\begin{ex}\label{EX103}
Application of Theorem \ref{T940}. \rm
By Facts \ref{T950}(b) and \ref{T941}
taking arbitrary summands $\BL_{i,j,k}$ from the class $\M _2\cup\Ord _{<\o^\o} \cup \Ord _{<\o^\o}^*$
we obtain finite axiomatizability, representation of models and VC$^\sharp$ for $\Th (\X)$,
where  $\X \cong \sum _{\BI }\prod _{j<m_i}\sum _{\BK_{i,j}}(\BL _{i,j,\,k})_r$.
In particular, we can take $\BL_{i,j,k}$-s from the class $\M _1$ of countable  $\o$-categorical linear orders.
\end{ex}

{\footnotesize

\end{document}